\RequirePackage{etoolbox}
\patchcmd{\bibliographystyle}{#1}{abbrvnat}{}{}
\documentclass[twoside,11pt,pdftex]{article}

\usepackage{jmlr2e}
\usepackage{amsmath} % not included in jmlr.sty ??
\usepackage{multirow}
\usepackage{color}
\usepackage[T1]{fontenc}
\usepackage[latin1]{inputenc}
\usepackage[pdftex,                %
bookmarks         = true,%     % Signets
bookmarksnumbered = true,%     % Signets numerotes
pdfpagemode       = None,%     % Signets/vignettes fermes a l'ouverture
pdfstartview      = FitH,%     % La page prend toute la largeur
pdfpagelayout     = SinglePage,% Vue par page
colorlinks        = FALSE,%     % Liens en couleur
urlcolor          = magenta,%  % Couleur des liens externes
pdfborder         = {0 0 0}%   % Style de bordure : ici, pas de bordure
]{hyperref}%                   % Utilisation de HyperTeX

\title{Asymptotic Results on Adaptive False Discovery Rate Controlling Procedures Based on Kernel Estimators}
\ShortHeadings{Asymptotics of Kernel-Based Adaptive FDR Controlling Procedures % Based on Kernel Estimators
}{P. Neuvial}
\firstpageno{1}

\author{
  \name Pierre Neuvial
  \email pierre.neuvial@genopole.cnrs.fr \\
  \addr Laboratoire de Probabilités et Modèles Aléatoires\\ 
  Université Paris VII-Denis Diderot\\ \medskip
  175 rue du Chevaleret, 75013 Paris, France \\ 
  INSERM, U900, Paris, F-75248 France\\
  Ecole des Mines de Paris, ParisTech, Fontainebleau, F-77300 France\\ \medskip
  Institut Curie, 26 rue d'Ulm, Paris cedex 05, F-75248 France\\   
  Current affiliation: Laboratoire Statistique et G\'enome\\
  Universit\'e d'\'Evry Val d'Essonne, UMR CNRS 8071 -- USC INRA, France
}
\editor{}

\graphicspath{{img/} {./}}

\newcommand{\expo}[1]{\exp{\left(#1\right)}}  % exponentielle entre parentheses

\DeclareMathOperator{\sgn}{sgn}
\newcommand{\e}[1]{e^{#1}}                    % exponentielle 'e puissance'
\newcommand{\Prob}{\mathbb{P}}                   % proba
\DeclareMathOperator{\Var}{Var}
\newcommand{\Exp}[1]{\mathbb{E}\left[#1\right]}       % expectation
\newcommand{\ind}[1]{\mathbf{1}_{#1}}         % indicatrice

\newcommand{\po}[1]{\mathrm{o}\left(#1\right)}   % petit o
\newcommand{\pop}[1]{\mathrm{o_P}\left(#1\right)}   % petit o
\newcommand{\hn}{\mathcal{H}_0}             % null hypothesis
\newcommand{\ha}{\mathcal{H}_1}             % alternative hypothesis
\newcommand{\eps}{\varepsilon}              % sparsity level

\newcommand{\tcdf}{F}                       % {c}df of the {t}est statistics
\newcommand{\tpdf}{f}                       % {p}df of the {t}est statistics

\newcommand{\tcdfn}{\tcdf_0}                % {c}df of the {t}est statistics under H_{0}
\newcommand{\tpdfn}{\tpdf_0}                % {p}df of the {t}est statistics under H_{0}
\newcommand{\tcdfa}{\tcdf_1}                % {c}df of the {t}est statistics under H_{1}
\newcommand{\tpdfa}{\tpdf_1}                % {p}df of the {t}est statistics under H_{1}

\newcommand{\pcdf}{G}                       % {c}df of the {p}values
\newcommand{\ppdf}{g}                       % {p}df of the {p}values

\newcommand{\pcdfn}{\pcdf_0}                % {c}df of the {p}values under H_{0} = identity !!
\newcommand{\pcdfa}{\pcdf_1}                % {c}df of the {p}values under H_{1}
\newcommand{\ppdfa}{\ppdf_1}                % {p}df of the {p}values under H_{1}

\newcommand{\ppdfats}{\ppdfa}                % {p}df of the two-sided {p}values under H_{1}
\newcommand{\ppdfaos}{\ppdfa}                % {p}df of the one-sided {p}values under H_{1}
\newcommand{\ppdfts}{\ppdf}                % {p}df of the two-sided {p}values 
\newcommand{\ppdfos}{\ppdf}                % {p}df of the one-sided {p}values 
\newcommand{\pcdfos}{\pcdf}                % {c}df of the one-sided {p}values 
\newcommand{\pcdfaos}{\pcdfa}                % {c}df of the one-sided {p}values 1

\newcommand{\empcdf}[2]{\widehat{\mathbb{#1}}_{#2}}  % empirical {c}df
\newcommand{\rempcdf}[2]{\bar{\mathbb{#1}}_{#2}}  % empirical {c}df

\newcommand{\ecdf}[1]{\empcdf{#1}{m}}  % empirical {c}df
\newcommand{\rcdf}[1]{\rempcdf{#1}{m}}  % recentered empirical {c}df

\newcommand{\pecdf}{\ecdf{\pcdf}}           % empirical {c}df of the {p}values
\newcommand{\prcdf}{\rcdf{\pcdf}}           % empirical {c}df of the {p}values

\newcommand{\pecdfa}{\empcdf{\pcdf}{1,m}}     % empirival {c}df of the {p}values under H_{1}
\newcommand{\pecdfn}{\empcdf{\pcdf}{0,m}}     % empirival {c}df of the {p}values under H_{0}

\newcommand{\pecdfxU}[1]{\widehat{\Gamma}_{#1,m}}     % empirival {c}df of the {p}values under H_{x} in the uncond model
\newcommand{\pecdfaU}{\pecdfxU{1}}     % empirival {c}df of the {p}values under H_{1} in the uncond model
\newcommand{\pecdfnU}{\pecdfxU{0}}     % empirival {c}df of the {p}values under H_{0s} in the uncond model

\newcommand{\lr}{\frac{\tpdfa}{\tpdfn}}
\newcommand{\ilr}{\tpdfa/\tpdfn} % inline version of \lr
\newcommand{\lrg}{\frac{\tpdfa^\gamma}{\tpdfn^\gamma}}
\newcommand{\ilrg}{\tpdfa^\gamma/\tpdfn^\gamma}

\newcommand{\as}{\alpha^\star}
\newcommand{\asPI}{\alpha^\star_0}
\newcommand{\bs}{\underline{\alpha}^\star}

\DeclareMathOperator{\FDR}{FDR}
\newcommand{\fdr}{\ensuremath{\FDR}}

\DeclareMathOperator{\FDP}{FDP}
\newcommand{\fdp}{\ensuremath{\FDP}}

\newcommand{\pim}{\ensuremath{\pi_{0}}}
\newcommand{\pih}{\hat{\pi}_{0,m}}
\newcommand{\pihSto}{\pih^\mathrm{Sto}}
\newcommand{\pihl}{\pihSto(\lambda)}

\newcommand{\pibSto}{\overline{\pi_0}\left(\lambda\right)}

\newcommand{\tBH}{\widehat{\tau}_m}

\newcommand{\tBHinf}{\tau_\infty}

\newcommand{\tSto}{\widehat{\tau}^{0,\lambda}_m}

\newcommand{\tStoinf}{\tau^{0,\lambda}_\infty}

\newcommand{\tPI}{\widehat{\tau}_m^0}
\newcommand{\rPI}{\widehat{\rho}_m^0}
\newcommand{\vPI}{\widehat{\nu}_m^0}

\newcommand{\fdpPI}{\widehat{\FDP}_m^0}

\newcommand{\tPIinf}{\tau_\infty^0}
\newcommand{\rPIinf}{\rho_\infty^0}
\newcommand{\vPIinf}{\nu_\infty^0}

\newenvironment{todo}{\color{red} [\bf TODO: \itshape }{]}

\newcommand{\Hh}{H\!h}

\begin{document}

\maketitle

%\noindent { \sf This version: \timestamp}
  
\abstract{
  The False Discovery Rate (FDR) is a commonly used type I error rate in multiple testing problems.  It is defined as the expected False Discovery Proportion (FDP), that is, the expected fraction of false positives among rejected hypotheses.  When the hypotheses are independent, the Benjamini-Hochberg procedure achieves FDR control at any pre-specified level.  By construction, FDR control offers no guarantee in terms of power, or type II error.  A number of alternative procedures have been developed, including plug-in procedures that aim at gaining power by incorporating an estimate of the proportion of true null hypotheses. 

In this paper, we study the asymptotic behavior of a class of plug-in procedures based on kernel estimators of the density of the $p$-values, as the number $m$ of tested hypotheses grows to infinity.  In a setting where the hypotheses tested are independent, we prove that these procedures are asymptotically more powerful in two respects: (i) a tighter asymptotic FDR control for any target FDR level and (ii) a broader range of target levels yielding positive asymptotic power.  We also show that this increased asymptotic power comes at the price of slower, non-parametric convergence rates for the FDP.  These rates are of the form $m^{-k/(2k+1)}$, where $k$ is determined by the regularity of the density of the $p$-value distribution, or, equivalently, of the test statistics distribution.  These results are applied to one- and two-sided tests statistics for Gaussian and Laplace location models, and for the Student model.
}

\begin{keywords} Multiple testing, False
  Discovery Rate, Benjamini Hochberg's procedure, power, criticality,
  plug-in procedures, adaptive control, test statistics distribution,
  convergence rates, kernel estimators. \end{keywords}

%\tableofcontents

\section{Introduction}
\label{sec:intro}

% \subsection{Background}
% \label{sec:background}

Multiple simultaneous hypothesis testing has become a major issue for high-dimensional data analysis in a variety of fields, including non-parametric estimation by wavelet methods in image analysis, functional magnetic resonance imaging (fMRI) in medicine, source detection in astronomy, and DNA microarray or high-throughput sequencing analyses in genomics.  Given a set of observations corresponding either to a null hypothesis or an alternative hypothesis, the goal of multiple testing is to infer which of them correspond to true alternatives.  This requires the definition of risk measures that are adapted to the large number of tests performed: typically  $10^4$ to $10^6$ in genomics.  
The False Discovery Rate (\fdr) introduced by~\cite{benjamini95controlling} is one of the most commonly used and one of the most widely studied such risk measure in large-scale multiple testing problems. The \fdr\ is defined as the expected proportion of false positives among rejected hypotheses.  A simple procedure called the Benjamini-Hochberg (BH) procedure provides \fdr\ control when the tested hypotheses are independent~\citep{benjamini95controlling} or follow specific types of positive dependence~\citep{benjamini01the-control}.  

When the hypotheses tested are independent, applying the BH procedure at level $\alpha$ in fact yields FDR $=\pi_0\alpha$, where $\pi_0$ is the unknown fraction of true null hypotheses~\citep{benjamini01the-control}.  This has motivated the development of a number of ``plug-in'' procedures, which consist in applying the BH procedure at level $\alpha/\hat{\pi}_0$, where $\hat{\pi}_0$ is an estimator of $\pi_0$.  A typical example is the Storey-$\lambda$ procedure~\citep{storey02daf,storey04strong} in which $\hat{\pi}_0$  is a function of the empirical cumulative distribution function of the $p$-values.

  In this paper, we consider an asymptotic framework where the number
  $m$ of tests performed goes to infinity. When $\hat{\pi}_0$
  converges in probability to $\pi_{0,\infty} \in [\pi_0, 1)$ as $m \to
  +\infty$, the corresponding plug-in procedure is by construction
  asymptotically more powerful than the BH procedure, while still
  providing FDR $\leq \alpha$. However, as FDR control only implies
  that the \emph{expected} FDP is below the target level, it is of
  interest to study the \emph{fluctuations} of the FDP achieved by
  such plug-in procedures around their corresponding FDR.
  This paper studies the influence of the plug-in step on the
  asymptotic properties of the corresponding procedure for a particular
  class of estimators of $\pi_0$, which may be written as kernel
  estimators of the density of the $p$-value distribution at 1.

\section{Background and notation}
\label{sec:background-notation}

\subsection{Settings}
\label{sec:model}
\paragraph{Testing one hypothesis.}

We consider a test statistic $X$ distributed as $\tcdfn$ under a null hypothesis $\hn$ and as $\tcdfa$ under an alternative hypothesis $\ha$.  We assume that  for $a \in \{0,1\}$, $\tcdf_a$ is continuously differentiable, and that the corresponding density function, which we denote by $\tpdf_a$, is positive.
This testing problem may be formulated in terms of $p$-values instead of test statistics. The $p$-value function is defined as $p:x \mapsto \Prob_{\hn}\left(X \geq x\right) = 1-\tcdfn(x)$ for one-sided tests and $p:x \mapsto  \Prob_{\hn}\left(|X| \geq |x|\right)$ for two-sided tests.  As $\tcdfn$ is continuous, the $p$-values are uniform on $[0,1]$ under $\hn$.  For consistency we denote by $\pcdfn$ the corresponding distribution function, that is, the identity function on $[0,1]$. Under $\ha$, the distribution function and density of the $p$-values are  denoted by $\pcdfa$ and $\ppdfa$, respectively.  
Their expression as functions of the distribution of the test statistics are recalled in Proposition~\ref{prop:p-values} below in the case of one- and two-sided $p$-values.  For two-sided $p$-values, we assume that the distribution function of the test statistics under $\hn$ is symmetric (around 0):
\begin{align}   \label{cond:symmetry}  \tag{Sym}
  \forall x \in \mathbb{R}, \tcdfn(x) + \tcdfn(-x) = 1\,.
\end{align}
Condition~\eqref{cond:symmetry} is typically fulfilled in usual models such as Gaussian or Laplace (double exponential) models.  Under \eqref{cond:symmetry}, the two-sided $p$-value satisfies $p(x) = 2\left(1-\tcdfn(|x|)\right)$ for any $x \in \mathbb{R}$.

\begin{proposition}[One- and two-sided $p$-values]
\label{prop:p-values} For $t \in [0,1]$, let $q_0(t) = \tcdfn^{-1}(1-t)$.  The distribution function $\pcdfa$ and the and density function $\ppdfa$ of the $p$-value under $\ha$ at $t$ satisfy the following:
\begin{enumerate}
\item for a one-sided $p$-value, $\pcdfa(t)  = 1-\tcdfa\left(q_0(t)\right)$ and $\ppdfa(t) = (\ilr)\left(q_0(t)\right)$;
\item for a two-sided $p$-value, $\pcdfa(t) = 1-\tcdfa\left(q_0(t/2)\right) + \tcdfa\left(-q_0(t/2)\right)$ and 

$\ppdfa(t) = 1/2\left((\ilr)\left(q_0(t/2)\right) + (\ilr)\left(-q_0(t/2)\right) \right)$.
\end{enumerate}
\end{proposition}
The assumption that $\tpdfa$ is positive entails that $\ppdfa$ is positive as well. 
We further assume that
\begin{align}
  \label{cond:concavity} \tag{Conc} 
  \pcdfa \textnormal{ is concave.}
\end{align}
As $\ppdfa$ is a function of the likelihood ratio $\ilr$ and the non-increasing function $q_0$, \eqref{cond:concavity} may be characterized as follows:
\begin{lemma}[Concavity and likelihood ratios]
\label{lm:concavity}
\begin{enumerate}
\item For a one-sided $p$-value, \eqref{cond:concavity} holds if and only if the likelihood ratio $\ilr$ is non-decreasing.
\item For a two-sided $p$-value  under \eqref{cond:symmetry}, \eqref{cond:concavity} holds if and only if $x \mapsto (\ilr)(x) + (\ilr)(-x)$ is non-decreasing on $\mathbb{R}_+$.
\end{enumerate}
\end{lemma}

\paragraph{Multiple testing setting.}

We consider a sequence of independent tests performed as described above and indexed by the set $\mathbb{N}^*$ of positive integers.  
We assume that either all of them are one-sided tests, or all of them are two-sided tests.
This sequence of tests is characterized by a sequence $(\mathbf{H},\mathbf{p})=(H_i, p_i)_{i \in \mathbb{N}^*}$
% $ \in (\{0, 1\} \times [0,1])^{\mathbb{N}^*}$
, where for each $i\in \mathbb{N}^*$, $p_i$ is a $p$-value associated to the $i^{th}$ test, and $H_i$ is a binary indicator defined by
\begin{displaymath}
   H_i= \begin{cases}
   0 \textnormal{ if $\hn$ is true for test $i$}\\
   1\textnormal{ if $\ha$ is true for test $i$}
  \end{cases}.
\end{displaymath}

We also let $m_0(m)=\sum_{i=1}^m (1-H_i)$, and $\pi_{0,m} = m_0(m)/m$.
%In the multiple testing setting originally proposed by \cite{benjamini95controlling}, $\mathbf{H}$ is deterministic and  $\mathbf{p}$ is a sequence of independent variables such that $p_i \sim \pcdfn$ when $H_i=0$.  
Following the terminology proposed by~\cite{roquain11exact}, we define the \emph{conditional setting} as the situation where  $\mathbf{H}$ is deterministic and  $\mathbf{p}$ is a sequence of independent random variables such that for $i \in \mathbb{N}^*, p_i \sim \pcdf_{H_i}$.  This is a particular case of the setting originally considered by  \cite{benjamini95controlling}, where no assumption was made on the distribution of $p_i$ when  $H_i=1$.  
In the present paper, we consider an \emph{unconditional setting} introduced by~\cite{efron01empirical}, which is also known as the ``random effects'' setting.  Specifically, $\mathbf{H}$ is a sequence of random indicators, independently and identically distributed as $\mathcal{B}(1-\pi_0)$, where $\pi_0 \in (0,1)$, and conditional on $\mathbf{H}$,  $\mathbf{p}$ follows the conditional  setting, that is, the $p$-values satisfy $p_i \vert H_i \sim \pcdf_{H_i}$.  This unconditional setting has been widely used in the multiple testing literature, see, e.g., \cite{storey03the-positive,genovese04a-stochastic,chi07on-the-performance}. In this setting, the $p$-values are independently, identically distributed as  $\pcdf = \pi_0\pcdfn + (1-\pi_0)\pcdfa$, and 
$m_0(m)$ follows the binomial distribution $Bin(m, \pi_0)$.

  \begin{remark}\label{rk:non-sparse}
    We are assuming that $\pi_0<1$, which implies that the proportion $1-\pi_{0,m}$
    of true null alternatives does not vanish as $m \to +\infty$.
    While this restriction is natural in the unconditional setting
    considered in this paper, we note that our results do not apply to the ``sparse'' situation where $\pi_{0,m} \to 1$ as $m \to +\infty$.
  \end{remark}

As $\pcdfn$ is the identity function, the multiple testing model is entirely characterized by the two parameters $\pi_0$ and $\pcdfa$ (or, equivalently, $\pi_0$ and $\pcdf$), where $\pcdfa$ is itself entirely characterized by $\tcdfn$ and $\tcdfa$, by Proposition~\ref{prop:p-values}.
 The mixture distribution $\pcdf$ is concave if and only if~\eqref{cond:concavity} holds.  More generally, we note that making a regularity assumption on $\pcdfa$ (or $\ppdfa$) is equivalent to making the same regularity assumption on $\pcdf$ (or $\ppdf$):

\begin{remark}[Differentiability assumptions]
  Throughout the paper, differentiability assumptions on the distribution of the $p$-values near 1 are expressed in terms of $\ppdf$, the (mixture) $p$-value density. As $\ppdf = \pi_0 + (1-\pi_0)\ppdfa$, we note that they could equally be written in terms of $\ppdfa$, the $p$-value density under the alternative hypothesis. 
% Specifically, for any $l\geq 1$, $\ppdf$ is $l$ times differentiable at 1 if and only if $\ppdfa$ is $l$ times differentiable at 1, and we have  $\ppdf^{(l)}(1) =  (1-\pi_0)\ppdfa^{(l)}(1)$.
\end{remark}

% \subsection{False Discovery Rate and power of multiple testing procedures}
\subsection{Type I and II error rate control in multiple testing}
\label{sec:mult-test-procs}
  We define a multiple testing procedure $\mathcal{P}$  as a
  collection of functions $(\mathcal{P}_\alpha)_{\alpha \in [0,1]}$
  such that for any $\alpha \in [0,1]$, $\mathcal{P}_\alpha$ takes as
  input a vector of $m$ $p$-values, and returns a subset of $\{1,
  \dots m\}$ corresponding to the indices of hypotheses to be
  rejected.  %%%
For a given procedure $\mathcal{P}$ and a given $\alpha \in [0,1]$, the function $\mathcal{P}_\alpha$ will be called ``Procedure $\mathcal{P}$ at (target) level $\alpha$''.
  In this paper, we focus on \emph{thresholding-based} multiple
  testing procedures, for which the rejected hypotheses are those with
  $p$-values less than a threshold.  Each possible value for the
  threshold corresponds to a trade-off between false positives (type I
  errors) and false negatives (type II errors).  
% \paragraph{False Discovery Rate and power}
Most risk measures developed for multiple testing procedures are based on type I errors.  We focus on one such measure, the False Discovery Rate (\fdr), which is one of the most widely used error rate in multiple testing.
Denoting by $R_m$ be the total number of rejections of  $\mathcal{P}_\alpha$ among $m$ hypotheses tested, and by $V_m$ the number of false rejections, the corresponding False Discovery Proportion is defined as $\FDP_m= V_m/(R_m\vee 1)$, and the False Discovery Rate is the expected FDP, that is:
\begin{eqnarray}
\label{eq:def-fdr}
\fdr_m = \Exp{\frac{V_m}{R_m\vee 1}}. 
\end{eqnarray}
A trivial way to control the \fdr\ --- or any risk measure only based on type I errors --- is to make no rejection with high probability.  Obviously, this is not the best strategy, as it may lead to a high number of type II errors.  
The performance of multiple testing procedures may be evaluated through their power, which is a function of the number of type II errors.  Specifically, the power of a multiple testing procedure at level $\alpha$ is generally defined as the (random) proportion of correct rejections (true positives)  among true alternative hypotheses, see, e.g., \cite{chi07on-the-performance}:
\begin{eqnarray}
\label{eq:def-power}
\Pi_m = \frac{R_m-V_m}{(m-m_0(m)) \vee 1}. 
\end{eqnarray}

\begin{remark}
  All of the quantities defined in this section implicitly depend on the multiple testing procedure considered, $\mathcal{P}=(\mathcal{P}_\alpha)_{\alpha \in [0,1]}$.  However,  for simplicity, we will write $R_m,V_m,\fdr_m$, and $\Pi_m$, instead of $R_m^{\mathcal{P}_\alpha},V_m^{\mathcal{P}_\alpha},\fdr_m^{\mathcal{P}_\alpha}$, and $\Pi_m^{\mathcal{P}_\alpha}$ whenever not ambiguous.
\end{remark}

  \begin{remark}[Power of thresholding-based procedures]
    \label{rk:increased-power}
    By definition, the power of a thresholding-based procedure is a
    non-decreasing function of its threshold. Therefore, among
    thresholding-based procedures that yield FDR less than a
    prescribed level, maximizing power is equivalent to maximizing the
    threshold of the procedure. 
  \end{remark} 

\subsection{The Benjamini-Hochberg  procedure}
\label{sec:bh-procedure}
Suppose we wish to control the \fdr\ at level $\alpha$.  Let $p_{(1)}\leq \ldots \leq p_{(m)}$ be the ordered $p$-values, and denote by $H_{(i)}$ the null hypothesis corresponding to $p_{(i)}$.  Define $\widehat{I}_m(\alpha)$ as the largest index $k\geq 0$ such that $p_{(k)}\leq \alpha k/m$.  The Benjamini-Hochberg  procedure at level $\alpha$ rejects all $H_{(i)}$ such that $i \leq \widehat{I}_m(\alpha)$ (if $\widehat{I}_m(\alpha)=0$, then no rejection is made).  This procedure has been proposed by \cite{benjamini95controlling} in the context of FDR control; \cite{seeger68a-note} reported that it had previously been used by~\cite{eklund61massignifikansproblemet} in another multiple testing context.
When all true null hypotheses are independent, the BH procedure at level $\alpha$ yields \emph{strong} FDR control, that is, it entails $\fdr \leq \alpha$ regardless of the number of true null hypotheses \citep{benjamini95controlling}.  The BH procedure also controls the \fdr\ when the $p$-values satisfy specific forms of positive dependence, see \citet{benjamini01the-control}.
Figure \ref{fig:bh} illustrates the application of the BH procedure with $\alpha=0.2$ to $m=100$ simulated hypotheses, among which 20 are true alternatives. The left panel illustrates the above definition of the BH procedure.  An equivalent definition is that the procedure rejects all hypotheses with associated \emph{p}-value is less than $\tBH(\alpha)=\alpha \widehat{I}_m(\alpha)/m$.  
\begin{figure}[!h]
  \centering
  \includegraphics[width=0.49\columnwidth]{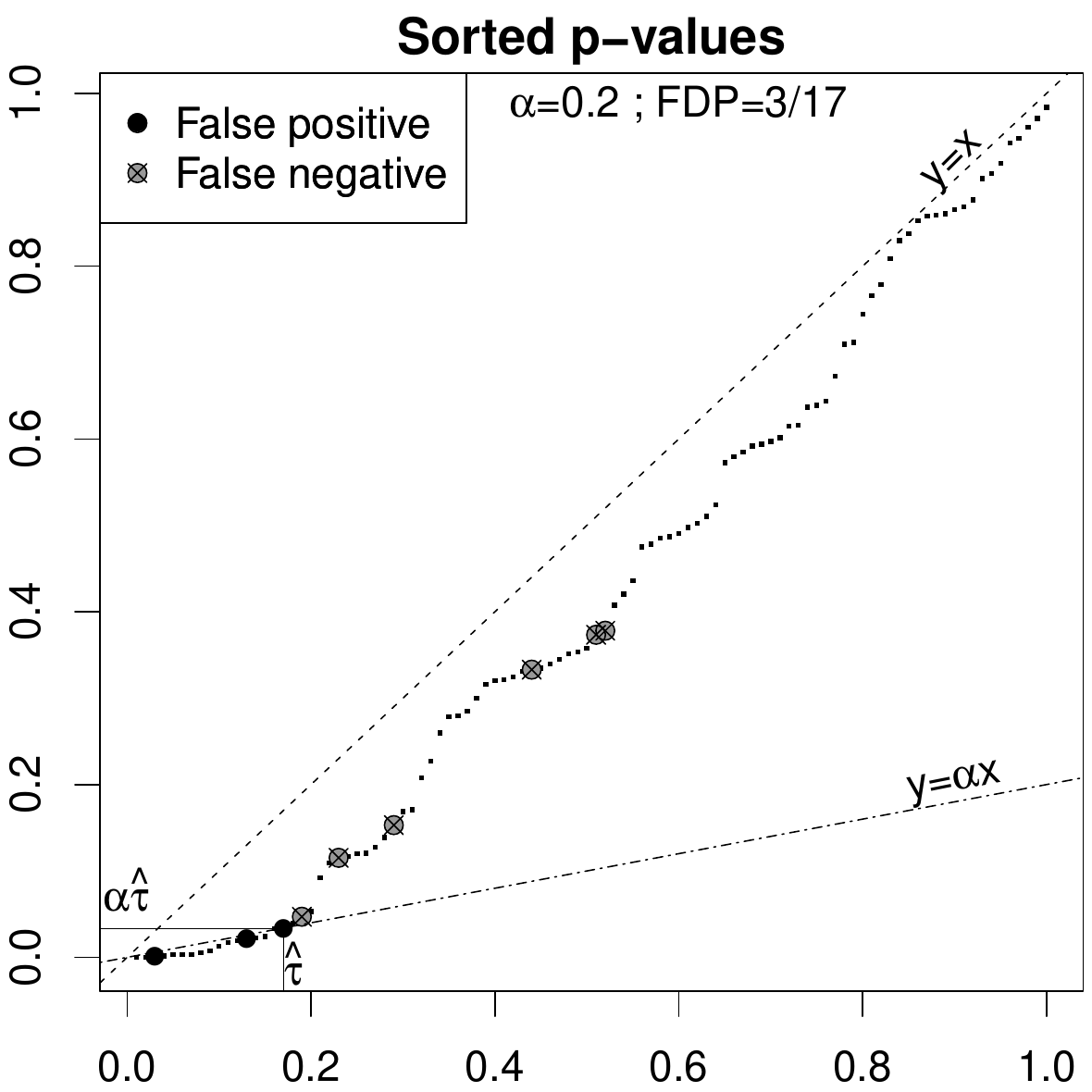}
  \caption{\emph{Illustrations of the BH procedure on a simulated example with $m=100$. Left: sorted $p$-values: $i/m \mapsto p_{(i)}$. Right: empirical distribution function: $t \mapsto \pecdf(t)$.}}
  \label{fig:bh}
\end{figure}
The right panel provides a dual representation of the same information, where the $x$ and $y$ axes have been swapped.  It gives a geometrical interpretation of $\tBH(\alpha)$ as the largest crossing point between the line $y=x/\alpha$ and the empirical distribution function of the $p$-values, defined for $t \in [0,1]$ by $\pecdf(t) =  \sum_{i=1}^m \ind{P_{i}\leq t}$:
\begin{eqnarray}\label{eq:tau-bh}
  \tBH(\alpha) = \sup \{t \in [0,1], \pecdf(t) \geq t/\alpha\}\,.
\end{eqnarray}

\subsection{Plug-in procedures}
% \paragraph{Plug-in procedures.}
In our setting where all of the hypotheses tested are independent, the BH procedure at target level $\alpha$ (henceforth denoted by BH$(\alpha)$ for short) in fact yields \fdr\ control at level $\pi_0 \alpha$ exactly~\citep{benjamini95controlling,benjamini01the-control}. This entails that the BH$(\alpha')$ procedure yields $\fdr \leq \alpha$ if and only if $\alpha' \leq \alpha/\pi_0$.  Therefore, as the threshold of the BH($\alpha$) procedure is a non-decreasing function of $\alpha$ and by Remark~\ref{rk:increased-power}, the BH$(\alpha/\pi_0)$ procedure is optimal in our setting, in the sense that it yields maximum power among procedures of the form  BH$(\alpha')$ that control the \fdr\ at level $\alpha$.  As $\pi_0$ is unknown, this procedure cannot be implemented; it is generally referred to as the Oracle BH procedure.  

\begin{remark}
  \label{rk:alpha-less-than-pi0}
  If $\alpha \geq \pi_0$, then rejecting all null hypotheses is
  optimal, as it corresponds to the largest possible threshold while
  still maintaining $\fdr =\pi_0 \leq \alpha$.  Therefore, we will
  assume that $\alpha < \pi_0$ throughout the paper.
\end{remark}

In order to mimic the Oracle procedure, it is natural to apply the BH procedure at level $\alpha/\pih$, where $\pih \leq 1$ is an estimator of $\pi_0$~\citep{benjamini00on-the-adaptive}. Such plug-in procedures (also known as two-stage adaptive procedures) have the same geometric interpretation as the BH procedure (see Figure~\ref{fig:bh}) in terms of the largest crossing point, with  $\alpha/\pih$ instead of $\alpha$. Their rejection threshold can be written as $\tPI(\alpha) = \tBH(\alpha/\pih)$, that is:
\begin{equation}
  \label{eq:tau-plug-in}
  \tPI(\alpha) = \sup \{t \in [0,1], {\pecdf(t)} \geq \pih t/\alpha\}\,.
\end{equation}
Note that $\tPI$ depends on the observations through both $\pecdf$ and $\pih$.  By construction, a plug-in procedure based on an estimator $\pih$ that converges in probability to $\pi_{0,\infty} \in [\pi_0, 1)$ as $m \to +\infty$ is asymptotically more powerful that the original BH procedure.

\paragraph{The Storey-$\lambda$ estimator.}
Adapting a method originally proposed by \cite{schweder82plots}, \cite{storey02daf} defined 
% $\pihl=\frac{\#\{i/P_i \geq \lambda\}}{\#\{i \geq \lambda\}}$ 
$\pihl=\#\{i/P_i \geq \lambda\}/\#\{i \geq \lambda\}$ for $\lambda \in (0,1)$.  This estimator may also be written as a function of the empirical distribution of the $p$-values: 
 
 \begin{equation}
   \pihl=\frac{1-\pecdf(\lambda)}{1-\lambda}\,.
   \label{eq:pi0-Storey}
\end{equation}
The rationale for $\pihl$ is that under~\eqref{cond:concavity}, larger $p$-values are more likely to correspond to true null hypotheses than smaller ones.  
Moreover, $\pihl$ converges in probability to $(1-\pcdf(\lambda))/(1-\lambda)$, where the limit is greater than $\pi_0$ as $\pcdf$ stochastically dominates the uniform distribution.  Several choices of $\lambda$ have been proposed, including $\lambda=1/2$~\citep{storey03statistical}, a data-driven choice based on the bootstrap~\cite{storey04strong}, and $\lambda=\alpha$~\citep{blanchard09adaptive}.  
In our setting, a slightly modified version of the corresponding plug-in BH$(\alpha/\pihl)$ procedure where $1/m$ is added to the numerator in \eqref{eq:pi0-Storey} achieves strong \fdr\ control at level $\alpha$~\citep{storey04strong}.
We note that the Storey-$\lambda$ estimator $\pihl$ can be viewed as a kernel estimator of the density $\ppdf$ at 1.  
\begin{definition}[Kernel of order $\ell$ and kernel estimator of a density at a point]
\label{def:kernel}
  \begin{enumerate}
  \item     A kernel of order $\ell \in \mathbb{N}$ is a function $K:\mathbb{R} \to \mathbb{R}$ such that the functions $u \mapsto u^j K(u)$ are integrable  for any $j=0 \ldots \ell$, and satisfy $\int_{\mathbb{R}}K=1$, and $\int_{\mathbb{R}}u^jK(u)du=0$ for $j=1 \dots \ell$.  
\item    The kernel estimator of a density $\ppdf$ at $x_0$ based on $m$ independent, identically distributed observations $x_1, \dots x_m$ from $\ppdf$ is defined by
    \begin{displaymath}
      \hat{\ppdf}_m(x_0) = \frac{1}{m h} \sum_{i=1}^mK\left(\frac{x_i-x_0}{h}\right)\,,
    \end{displaymath}
    where $h>0$ is called the bandwidth of the estimator, and $K$ is a kernel.
  \end{enumerate}
\end{definition}
By Definition~\ref{def:kernel}, $\pihl$ is a kernel estimator of the density $\ppdf$ at 1 with kernel $K^{\rm Sto}(t) =\ind{[-1, 0]}(t)$ and bandwidth $h=1-\lambda$. $K^{\rm Sto}$ is an asymmetric, rectangular kernel of order 0.  

\subsection{Criticality and asymptotic properties of FDR controlling procedures}
\label{sec:crit-asympt-prop}
Upper bounds on the asymptotic number of rejections of \fdr\ controlling procedures have been identified and characterized by \cite{chi07on-the-performance} and \cite{chi08positive}, who introduced the notion of  \emph{critical value of a multiple testing problem} and that of \emph{critical value of a multiple testing procedure}. 
Both notions are defined formally below.  They are tightly connected, with the important difference that the former only depends on the multiple testing problem, while the latter depends on both the multiple testing problem and a specific multiple testing procedure.

\begin{definition}[Critical value of a multiple testing problem \citep{chi07on-the-performance}]
\label{def:criticality}
The \emph{critical value} of the multiple testing problem parametrized
  by $\pi_0$ and $\pcdf$ is defined by 
\begin{eqnarray}
\label{eq:def-critic}
  \bs = \inf_{t \in (0,1]} \frac{\pi_0 t}{\pcdf(t)}\,.
\end{eqnarray}
\end{definition}
\citet[proof of Proposition 3.2]{chi08positive} proved that for any multiple testing procedure, for $\alpha < \bs$, there exists a positive constant $c(\alpha)$ such that almost surely, for $m$ large enough, the events $\{V_m/R_m\leq \alpha \}$  and $\{R_m \geq c(\alpha) \log m\}$ are incompatible.  This restriction is intrinsic to the multiple testing problem, in the sense that it holds regardless of the considered multiple testing procedure. 
Obviously, this is not a limitation when $\bs=0$.  We introduce the following Condition:
\begin{align}
  \label{cond:critic} \tag{Critic} 
  \bs>0\,.
\end{align}
Whether \eqref{cond:critic} is satisfied or not only depends on $\pcdf$.  However, the value of $\bs$ as defined in \eqref{eq:def-critic} depends on both  $\pi_0$ and $\pcdf$.   Under~\eqref{cond:concavity} we have
  $\bs = \lim_{t \to 0} \pi_0 t/\pcdf(t) =
  \pi_0/(\pi_0+(1-\pi_0)\ppdfa(0))$, where $\ppdfa(0) \in [0,+\infty]$
  is defined by $\ppdfa(0) = \lim_{t\to 0} \ppdfa(t)$.  By
  Proposition~\ref{prop:p-values}, $\ppdfa(0)$ only depends on the
  behavior of the test statistics distribution.  In particular, under~\eqref{cond:concavity}, \eqref{cond:critic} is satisfied if and only if
  the likelihood ratio $\ilr$ is bounded near $+\infty$.

We now introduce the notion of \emph{critical value of a multiple testing procedure}. \cite{chi07on-the-performance} defined the critical value of the BH procedure as $ \as_{BH} = \inf_{t \in (0,1]} t/\pcdf(t)$.  Let us denote by
\begin{equation}
  \tBHinf(\alpha) = \sup \{t \in [0,1], \pcdf(t) \geq t/\alpha \} 
  \label{eq:asymptotic-threshold-BH}
\end{equation}
 the rightmost crossing point between $\pcdf$ and the line $y=x/\alpha$.  \cite{chi07on-the-performance} has proved the following result:
\begin{proposition}[Asymptotic properties of the BH procedure~\citep{chi07on-the-performance}]
\label{prop:criticality-BH} 
For $\alpha \in [0,1]$, let $\tBH(\alpha)$ be the threshold of the BH$(\alpha)$ procedure, and let $\tBHinf(\alpha)$ be defined by \eqref{eq:asymptotic-threshold-BH}.  Let $ \as_{BH} = \inf_{t \in (0,1]} t/\pcdf(t)$.  As $m \to +\infty$,
\begin{enumerate}
\item If $\alpha< \as_{BH}$,  then $\tBH(\alpha) \overset{a.s.}{\to} 0$;
\item If  $\alpha > \as_{BH}$, then  $\tBH(\alpha) \overset{a.s.}{\to} \tBHinf(\alpha)$, where the limit is positive.
\end{enumerate}
 \end{proposition}
A straightforward consequence of Proposition~\ref{prop:criticality-BH} is that the BH($\alpha$) procedure has asymptotically null power when $\alpha < \as_{BH}$ and positive power when $\alpha > \as_{BH}$.  The following Definition generalizes the notion of critical value of to a generic multiple testing procedure: 
  \begin{definition}[Critical value of a multiple testing procedure]
    \label{def:critical-value-proc}
    Let $\mathcal{P}=(\mathcal{P}(\alpha))_{\alpha \in [0,1]}$ denote a multiple testing procedure.    The critical value of $\mathcal{P}$ is defined by
    \begin{eqnarray}
      \label{eq:def-critical-value}
      \as_\mathcal{P} = \sup \left\{ \alpha \in [0,1],  
        \Pi_{m}^{\mathcal{P}(\alpha)} \overset{a.s.}{\underset{m\to +\infty}{\longrightarrow}} 0 \right\}\,.
    \end{eqnarray}
  \end{definition}
The critical value $\as_\mathcal{P} $ depends on both the procedure $\mathcal{P}$, and the multiple setting.  
For the BH procedure, criticality ($\alpha < \as_{BH}$) corresponds to situations where the target \fdr\ level  $\alpha$ is so small that there is no positive crossing point between $\pcdf$ and the line $y=x/\alpha$.  Conversely, when $\alpha > \as_{BH}$, there is a positive crossing point between  $\pcdf$ and the line $y=x/\alpha$, as illustrated by  Figure~\ref{fig:bh} (right).
The almost sure convergence results of Proposition~\ref{prop:criticality-BH} in the case $\alpha > \as_{BH}$ were extended by \cite{neuvial08asymptotic}, in the conditional setting.  Specifically, the threshold $\tBH(\alpha)$ of the BH procedure was shown to converge in distribution to $\tBHinf(\alpha)$ at rate $m^{-1/2}$ as soon as $\alpha > \as_{BH}$.  
\cite{neuvial08asymptotic} also proved that similar central limit theorems hold for a class of thresholding-based \fdr\ controlling procedures that covers some plug-in procedures, including the Storey-$\lambda$ procedure: the threshold of a procedure $\mathcal{P}$ of this class converges in distribution to a procedure-specific, positive value at rate $m^{-1/2}$ as soon as $\alpha > \as_{\mathcal{P}}$.

\paragraph{Criticality of a multiple testing problem and criticality of a procedure.}
Whether \eqref{cond:critic} hols or not only depends on the behavior of the test statistics distribution.  However, this condition is tightly connected to the critical value of FDR controlling procedures.  In order to shed some light on this connection, we note that $\bs=\pi_0 \as_{BH}$ may be interpreted as  the critical value of the Oracle BH procedure BH$(\alpha/\pi_0)$. 
Therefore, as the Oracle BH procedure at level $\alpha$ is the most powerful procedure among thresholding-based procedures that control FDR at level $\alpha$, $\bs$ is a lower bound on the critical values of these procedures.  Specifically, multiple problems for which \eqref{cond:critic} is satisfied or  not differ in that:
\begin{itemize}
\item when~\eqref{cond:critic} is satisfied, all thresholding-based procedures that control FDR have null asymptotic power in a range of levels containing $[0, \bs)$;
\item when~\eqref{cond:critic} is not satisfied, some procedures (including BH) have positive asymptotic power for any positive level $\alpha$.
\end{itemize}

\subsection*{Organization of the paper}
\label{sec:organization-paper}
This paper extends the asymptotic results of \cite{chi07on-the-performance} and \cite{neuvial08asymptotic} to the case of plug-in procedures of the form BH$(\alpha/\pih)$, where $\pih$ is a kernel estimator of the $p$-value distribution $\ppdf$ at 1.  Specifically, we consider a class of kernel estimators of $\pi_0$, which includes a modification of the Storey-$\lambda$ estimator, where the parameter $\lambda$ tends to 1 as $m \to \infty$.  
In Section \ref{sec:conv-rates-estim}, we prove that this class of estimators of $\pi_0$ achieves non-parametric convergence rates of the form $m^{-k/(2k+1)}/\eta_m$, where $\eta_m$ goes to 0 slowly enough as $m \to +\infty$, and $k$ controls the regularity of $\ppdf$ at 1.  In Section \ref{sec:conv-rate-cons}, we characterize the critical value $\asPI$ of plug-in procedures based on such estimators, and prove that when the target FDR level $\alpha$ % or the target FDR level $\alpha$, such that: for $\alpha<\asPI$, the plug-in procedures based on such estimators have asymptotically null power;  
is greater than $\asPI$, the convergence rate of these plug-in procedures is $m^{-k/(2k+1)}/\eta_m$, which is slower than the parametric rate achieved by the BH procedure and by the plug-in procedures studied in \cite{neuvial08asymptotic}.  
In Section~\ref{sec:exampl-locat-models}, these results are applied to one and two-sided tests in location and Student models. Practical consequences and possible extensions of this work are discussed in Section~\ref{sec:discussion}.

% In Section~\ref{sec:conv-rates-estim}, we establish asymptotic properties of for estimators of $\pi_0$ of the form $\pihSto(1-h_m)$, where $h_m \to 0$ as $m \to +\infty$, and for more general kernel estimators.  Specifically, we establish non-parametric convergence rates for these estimators as a function of the regularity of $\ppdf$ at 1.  In Section~\ref{sec:conv-rate-cons}, we show that the same convergence rates hold for the corresponding plug-in procedures.  In Section \ref{sec:exampl-locat-models}, these results are applied to one and two-sided Gaussian and Laplace location models.

\section{Asymptotic properties of non-parametric estimators of $\pi_0$}
\label{sec:conv-rates-estim}
Let $\lambda \in (0,1)$.  The expectation $\pibSto$ of the Storey-$\lambda$ estimator is given by
\begin{equation}
  \label{eq:pibl}
   \pibSto = \pi_0 + (1-\pi_0) \frac{1-\pcdfa(\lambda)}{1-\lambda}.
\end{equation}
Moreover, as a regular function of the empirical distribution of the $p$-values, $\pihl$ has the following asymptotic distribution for $\lambda \in (0,1)$~\citep{genovese04a-stochastic}:
\begin{equation}
\label{eq:asy-dist-storey}  
\sqrt{m} \left(\pihl - \pibSto \right) \rightsquigarrow \mathcal{N}\left(0, \frac{\pcdf(\lambda)(1-\pcdf(\lambda))}{(1-\lambda)^2}\right)\,.
\end{equation}

In our setting, $\ppdfa$ is positive, as noted in Section~\ref{sec:model}.  Therefore, we have $\pcdfa(\lambda)<1$ for any $\lambda \in (0,1)$, and the bias $\pibSto-\pi_0$ is  positive: the Storey-$\lambda$ estimator achieves a parametric convergence rate, but it is not a consistent estimator of $\pi_0$.
Under~\eqref{cond:concavity}, this bias decreases as $\lambda$ increases (by Equation \eqref{eq:pibl}).
%, but the asymptotic variance of $\pihl$ in Equation\eqref{eq:asy-dist-storey} goes to infinity as $\lambda \to 1$.  
In order to mimic the Oracle BH$(\alpha/\pi_0)$ procedure, it is therefore natural to choose $\lambda$ close to 1.  We consider plug-in procedures where $\pi_0$ is estimated by $\pihSto(1-h_m)$, with $h_m \to 0$ as $m \to +\infty$.   As the limit in probability of this estimator is $\ppdf(1) = \pi_0 + (1-\pi_0) \ppdfa(1)$, it is consistent if and only if the following ``purity'' condition, which has been introduced by~\cite{genovese04a-stochastic}, is met:
\begin{align}
  \label{cond:purity} \tag{Pur}
  \ppdfa(1)=0
\end{align}
We note that the Storey-$\lambda$ estimator is not a consistent estimator of $\pi_0$ even in when \eqref{cond:purity} is met.  Moreover, \eqref{cond:purity}  is entirely determined by the shape of the test statistics under the alternative hypothesis. 
The asymptotic bias and variance of $\pihSto(1-h_m)$ are characterized by  Proposition~\ref{prop:asy-var-pi0h}: 

\begin{proposition}[Asymptotic bias and variance of $\pihSto(1-h_m)$]
\label{prop:asy-var-pi0h}
Let $h_m$ be a positive sequence such that $h_m \to 0$.
\begin{enumerate}
\item If $m h_m \to +\infty$ as $m \to +\infty$, then 
\begin{displaymath}
  \sqrt{m h_m} \left(\pihSto(1-h_m) - \Exp{\pihSto(1-h_m)}\right) \rightsquigarrow \mathcal{N}(0, \ppdf(1))\,.
\end{displaymath}
\item Assume that for $k \geq 1$, $\ppdf$ is $k$ times differentiable
  at $1$, with $\ppdf^{(l)}(1) = 0$ for $1\leq l < k$. Then
  \begin{displaymath}
    \Exp{\pihSto(1-h_m)} - \ppdf(1) \underset{m \to +\infty}{=} 
    \frac{(-1)^{k}\ppdf^{(k)}(1)}{(k+1)!}h_m^k  + \po{h_m^k}\,.
  \end{displaymath}
\end{enumerate}
\end{proposition}
Only the bias term in Proposition~\ref{prop:asy-var-pi0h} depends on the regularity $k$ of the distribution near 1: the asymptotic bias is of order $h_m^k$, while the asymptotic variance of $\pihSto(1-h_m)$ is of order $(m h_m)^{-1}$, regardless of the regularity of the distribution.  
The bandwidth $h_m$ in Proposition~\ref{prop:asy-var-pi0h} realizes a trade-off between the asymptotic bias and variance of $\pihSto(1-h_m)$.   When the regularity of the distribution is known, a natural way to resolve this bias/variance trade-off is to calibrate $h_m$ such that the Mean Squared Error (MSE) of the corresponding estimator is asymptotically minimum.  This gives rise to an optimal choice of the bandwidth, which is characterized by the following proposition:
\begin{proposition}[Asymptotic properties of $\pihSto(1-h_m)$]
    \label{prop:opt-bw-sto}
    Assume that $\ppdf$ is $k$ times differentiable at $1$ for $k\geq
    1$, with $\ppdf^{(l)}(1)=0$ for $1 \leq l < k$.
    \begin{enumerate}
    \item If $\ppdf^{(k)}(1) \neq 0$, then the asymptotically optimal
      bandwidth for $\pihSto(1-h_m)$ in terms of MSE is of order
      $m^{-1/(2k+1)}$, and the corresponding MSE is of order
      $m^{-2k/(2k+1)}$. 
    \item Let $\eta_m$ be any sequence such that $\eta_m \to 0$ and
      $m^{k/(2k+1)}\eta_m \to +\infty$ as $m \to +\infty$.  Then,
      letting $h_m(k) =m^{-1/(2k+1)} \eta_m^2$, we have, as $m \to
      +\infty$:
      \begin{align}\label{eq:opt-rate-sto}
        m^{k/(2k+1)}\eta_m \left(\pihSto(1-h_m(k))-\ppdf(1)\right)
        \rightsquigarrow \mathcal{N}(0, \ppdf(1))\,
      \end{align}
    \end{enumerate}
  \end{proposition}
Proposition~\ref{prop:opt-bw-sto} is proved in Appendix~\ref{sec:conv-rate-kernel}. The convergence rate in \eqref{eq:opt-rate-sto} is a typical convergence rate for non-parametric estimators of a density at a point.  However,  Proposition~\ref{prop:opt-bw-sto} cannot be derived from classical results on kernel estimators (e.g. \cite{tsybakov03introduction}) as such results typically require that the order of the kernel matches the regularity $k$ of the density, whereas the kernel of Storey's estimator, $K^{\rm Sto}(t) = \ind{[-1, 0]}(t)$, is of order 0.
The results that can be obtained with kernels of order $k$ are summarized by Proposition~\ref{prop:opt-bw-k}; we refer to~\citet{tsybakov03introduction} for a proof of this result.

\begin{proposition}[$k^{\rm th}$ order kernel estimator \citep{tsybakov03introduction}]
\label{prop:opt-bw-k}
  Assume that for $k\geq 1$, $\ppdf$ is $k$ times differentiable at $1$. Let $\hat{\ppdf}_{m}^{k}(1)$ be a kernel estimator of $\ppdf(1)$ with bandwidth $h_m$, associated with a $k^{\rm th}$ order kernel. 
\begin{enumerate}
    \item The optimal bandwidth for $\hat{\ppdf}_{m}^{k}(1)$ in terms
      of MSE is of order $m^{-1/(2k+1)}$, and the
      corresponding MSE is of order $m^{-2k/(2k+1)}$;
    \item Let $\eta_m$ be any sequence such that $\eta_m \to 0$ and
      $m^{k/(2k+1)}\eta_m \to +\infty$ as $m \to +\infty$.  Then
      letting $h_m(k) = m^{-1/(2k+1)}\eta_m^2$, we have, as $m \to
      +\infty$:
      \begin{displaymath}
        m^{k/(2k+1)}\eta_m \left(\hat{\ppdf}_{m}^{k}(1)-\ppdf(1)\right) \rightsquigarrow \mathcal{N}(0,\ppdf(1))\,.
      \end{displaymath}
    \end{enumerate}
\end{proposition}

Propositions~\ref{prop:opt-bw-sto} and~\ref{prop:opt-bw-k} show that the convergence rate of kernel estimators of $\ppdf(1)$ with asymptotically optimal bandwidth directly depends on the regularity $k$ of $\ppdf$ at $1$.  The only difference between the two propositions is that the assumption that the first $k-1$ derivatives of $\ppdf$ are null at 1 for $\pih(1-h_m)$ is not needed for $k^{\mathrm{th}}$ order kernel estimators.  
 Importantly, these convergence rates cannot be improved in our setting, in the sense that $m^{-k/(2k+1)}$ is the minimax rate for the estimation of a density at a point where its regularity is of order $k$ \cite[Chapter 2]{tsybakov03introduction}. 

\paragraph{Connection to previously proposed estimators.} 
To the best of our knowledge, the only non-parametric estimators of $\pi_0$ for which convergence rates have been established in our setting are those proposed by  \cite{storey02daf}, \cite{swanepoel99the-limiting} and \cite{hengartner95fsc}. 
%The use of these estimators in the context of multiple testing  has been discussed by \cite{genovese04a-stochastic}. 
We now briefly review asymptotic properties of these estimators in the context of multiple testing, as stated in \cite{genovese04a-stochastic}, and show that their convergence rates can essentially be recovered by Propositions~\ref{prop:opt-bw-sto} and~\ref{prop:opt-bw-k}.

\begin{description}
\item[Confidence envelopes for the density:]
\cite{hengartner95fsc} derived a finite sample confidence envelope for a monotone density. Assuming that $\pcdf$ is concave and that $\ppdf$ is Lipschitz in a neighborhood of $1$, \cite{genovese04a-stochastic} obtained an estimator which 
converged to $\ppdf(1)$ at rate $(\ln{m})^{1/3}m ^{-1/3}$.  The same rate of convergence can be achieved by Proposition~\ref{prop:opt-bw-sto} or~\ref{prop:opt-bw-k} (for $\eta_m = (\ln{m})^{-1/3}$) if we assume that $\ppdf$ is differentiable at 1.  This is a slightly stronger assumption than the ones made by~\cite{hengartner95fsc}, but it still corresponds to a regularity of order 1.
\item[Spacings-based estimator:] 
 \cite{swanepoel99the-limiting} proposed a two-step estimator of the minimum of an unknown density based on the distribution of the spacings between observations: first, the location of the minimum is estimated, and then the density at this point is itself estimated. Assuming that at the value at which the density $\ppdf$ achieves its minimum, $\ppdf$ and $\ppdf^{(1)}$ are null, and $\ppdf^{(2)}$ is bounded away from $0$ and $+\infty$ and Lipschitz,  then for any $\delta>0$, there exists an estimator converging at rate $(\ln{m})^\delta m^{-2/5}$ to the true minimum.
% The corresponding estimator is denoted by $\pih^{\mathrm{Sw}}$.\\
The same rate of convergence can be achieved by Proposition~\ref{prop:opt-bw-sto} or~\ref{prop:opt-bw-k} (for $\eta_m= (\ln{m})^{-\delta}$) if one assumes that $\ppdf$ is twice differentiable at 1 (and additionally that $\ppdf^{(1)}(1)=0$ for  Proposition~\ref{prop:opt-bw-sto}).  In our setting, the Lipschitz condition for the second derivative is unnecessary: the minimum of $\ppdf$ is necessarily achieved at $1$ because $\ppdf$ is non-increasing (under~\eqref{cond:concavity}), so the first step of the estimation in~\cite{swanepoel99the-limiting} may be omitted.  
\end{description}
As both estimators are estimators of $\ppdf(1)$, the differences in their asymptotic properties are driven by the differences in the regularity assumptions made for $\ppdf$ (or $\ppdfa$) near 1, rather than by their specific form. 

\section{Consistency, criticality and convergence rates of plug-in procedures}
\label{sec:conv-rate-cons}

The aim of this section is to derive convergence rates for plug-in procedures based on the estimators $\pih$ of $\pi_0$ studied in Section~\ref{sec:conv-rates-estim}.  Specifically, our goal is to establish central limit theorems for the threshold $\tPI(\alpha)$ of the plug-in procedure BH$(\alpha/\pih)$ and the associated False Discovery Proportion, which we denote by $\fdp_m(\tPI(\alpha))$.  The convergence results obtained by \cite{neuvial08asymptotic} cover a broad class of FDR controlling procedures, including the BH procedure and plug-in procedures based on estimators of $\pim$ that depend on the observations only through the empirical distribution function $\pecdf$ of the $p$-values \citep{storey02daf,storey04strong,benjamini06adaptive}.  Although these results were obtained in the conditional setting of \citet{benjamini95controlling}, extending them to the unconditional setting considered here is relatively straightforward, because the proof techniques developed in \cite{neuvial08asymptotic} can be adapted to this setting. For completeness, the asymptotic properties of the BH procedure and the plug-in procedure based on the Storey-$\lambda$ estimator are derived in Appendix~\ref{sec:extens-uncond-setting}.  The problem considered in this section is more challenging, as the kernel estimators introduced in Section~\ref{sec:conv-rates-estim} depend on  $m$ not only through $\pecdf$, but also through the bandwidth of the kernel (e.g. $h_m$ for $\pihSto(1-h_m)$).

Let $\pih$  denote a generic estimator of $\pi_0$.  We assume that $\pih$ converges in probability to $\pi_{0,\infty} \leq 1$ as $m \to +\infty$.  We do not assume that $\pi_{0,\infty} = \pi_0$.  Therefore, $\pih$ may or may not be a consistent estimator of $\pi_0$.  We recall that the BH$(\alpha/\pih)$ procedure rejects all hypotheses with $p$-values smaller than
\begin{equation*}
  \tPI(\alpha) = \sup\left\{t \in [0,1], \pecdf(t) \geq \pih t/\alpha\right\}\,.
\end{equation*}
We now study the behavior of the  BH$(\alpha/\pih)$ procedure when  $\pih$ converges at a rate $r_m$ slower than the parametric rate $m^{-1/2}$ (i.e., $m^{-1/2} = \po{r_m}$).  
We define the asymptotic threshold $\tPIinf(\alpha)$ corresponding to $\tPI(\alpha)$ as 
\begin{equation}
  \label{eq:asy-tau-plug-in}
  \tPIinf(\alpha) = \sup\left\{t \in [0,1], \pcdf(t) \geq \pi_{0,\infty} t/\alpha\right\}\,.
\end{equation}
We have $\tPIinf(\alpha) = \tBHinf(\alpha/\pi_{0,\infty})$, that is, the asymptotic threshold of the  BH procedure defined in Equation~\eqref{eq:asymptotic-threshold-BH}  at level $\alpha/\pi_{0,\infty}$.

\begin{theorem}[Asymptotic properties of plug-in procedures]
\label{thm:asy-prop-plug-in}
Let $\pih$ be an estimator of $\pi_0$  such that $\pih \to \pi_{0,\infty}$ in probability as $m \to +\infty$.
Let $\asPI=\pi_{0,\infty}  \as_{BH}$. Then: 
\begin{enumerate}
  \item $\asPI$ is the critical value of the BH$(\alpha/\pih)$
    procedure;
\item Further assume that the asymptotic distribution of $\pih$ is given by
\begin{displaymath}
  \sqrt{m h_m} \left(\pih - \pi_{0,\infty}\right) \rightsquigarrow \mathcal{N}(0, s_0^2)\,
\end{displaymath}
for some $s_0$, with $h_m = \po{1/\ln\ln m}$ and $mh_m \to +\infty$ as $m \to +\infty$. Then,  under~\eqref{cond:concavity}, for any $\alpha > \asPI$,
  \begin{enumerate}
  \item The asymptotic distribution of the threshold $\tPI(\alpha)$ is
    given by
    \begin{displaymath}
      \sqrt{mh_m} \left(\tPI(\alpha) - \tPIinf(\alpha)\right) \rightsquigarrow \mathcal{N}\left(0, \left(\frac{s_0 \tPIinf(\alpha)/\alpha}{\pi_{0,\infty}/\alpha - \ppdf(\tPIinf(\alpha))}\right)^2\right)
    \end{displaymath}
  \item The asymptotic distribution of the FDP achieved by the
    BH$(\alpha/\pih)$ procedure is given by
    \begin{displaymath}
      \sqrt{mh_m}\left(\fdp_m(\tPI(\alpha))-\frac{\pi_0\alpha}{\pi_{0, \infty}}\right) \rightsquigarrow \mathcal{N}\left(0, \left(\frac{\pi_0\alpha s_0}{\pi_{0,\infty}^2}\right)^2 \right)\,.
    \end{displaymath}
  \end{enumerate}
\end{enumerate}
\end{theorem}

Theorem~\ref{thm:asy-prop-plug-in} states that for  $\alpha>\asPI$, for any estimator $\pih$ that converges in distribution at a rate $r_m$ slower than the parametric rate $m^{-1/2}$, the plug-in procedure BH$(\alpha/\pih)$ converges at rate 
$r_m$ as well.  This is a consequence of the fact that $r_m$ dominates the fluctuations of $\pecdf$, which are of parametric order.

We now state the main result of the paper (Corollary~\ref{cor:fdp-opt-bw}), that is, the asymptotic properties of plug-in procedures associated with the estimators of $\pi_0$ studied in Section~\ref{sec:conv-rates-estim}, for which $s_0^2=\ppdf(1)$.  This result can be derived by combining the results of Theorem~\ref{thm:asy-prop-plug-in} with those of Propositions~\ref{prop:opt-bw-sto} and~\ref{prop:opt-bw-k}.

  \begin{corollary}
    \label{cor:fdp-opt-bw}
    Assume that \eqref{cond:concavity} holds, and
    that $\ppdf$ is $k$ times differentiable at $1$ for $k \geq 1$.
    Define $h_m(k)=m^{-1/(2k+1)}\eta_m^2$, where $\eta_m \to 0$ and
    $m^{k/(2k+1)}\eta_m \to +\infty$ as $m \to +\infty$.  Denote by
    $\pih^k$ one of the following two estimators of $\pi_0$:
    \begin{itemize}
    \item Storey's estimator $\pihSto(1-h_m(k))$; in this case, it is
      further assumed that $\ppdf^{(l)}(1)=0$ for $1 \leq l < k$;
    \item A kernel estimator of $\ppdf(1)$ associated with a $k^{\rm
        th}$ order kernel with bandwidth $h_m(k)$.
    \end{itemize}
    Then 
    \begin{enumerate}
    \item $\asPI=\ppdf(1) \as_{BH}$ is the critical value of the
    BH$(\alpha/\pih^k)$ procedure;
    \item For any $\alpha > \asPI$,
      \begin{enumerate}
      \item The asymptotic distribution of the threshold
        $\tPI(\alpha)$ is given by
        \begin{displaymath}
          m^{k/(2k+1)}\eta_m \left(\tPI(\alpha) - \tPIinf(\alpha)\right)  \rightsquigarrow \mathcal{N}\left(0, \left(\frac{\tPIinf(\alpha)/\alpha}{\ppdf(1)/\alpha - \ppdf(\tPIinf(\alpha))}\right)^2 \ppdf(1) \right)
        \end{displaymath}
      \item The asymptotic distribution of the FDP achieved by the
        BH$(\alpha/\pih)$ procedure is given by
        \begin{displaymath}
          m^{k/(2k+1)}\eta_m \left(\fdp_m(\tPI(\alpha))-\frac{\pi_0\alpha}{\ppdf(1)}\right) \rightsquigarrow  \mathcal{N}\left(0, \frac{\pi_0^2\alpha^2}{\ppdf(1)^3}\right)\,. 
        \end{displaymath}
      \end{enumerate}
    \end{enumerate}
  \end{corollary}
  We note that unlike the modification of the Storey-$\lambda$ estimator studied here, the estimators of $\pi_0$ based on kernels of order $k$ do not require the first $k-1$ derivatives of $g$ at 1 to be null.  Therefore, the latter are generally preferable to the former.   
Corollary~\ref{cor:fdp-opt-bw} has the following consequences, which are also summarized in Table~\ref{tab:asymptotic-variances}:
  \begin{itemize}
  \item Assume that~\eqref{cond:purity} is met.  Then the asymptotic
    threshold of the BH$(\alpha/\pih)$ procedure is
    $\tBHinf(\alpha/\pi_0)$, that is, the asymptotic threshold of the
    Oracle procedure BH$(\alpha/\pi_0)$.  In particular, the
    asymptotic \fdp\ achieved by the estimators in
    Corollary~\ref{cor:fdp-opt-bw} is then \emph{exactly} $\alpha$
    (and its asymptotic variance is $\alpha^2/\pi_0$), whereas the
    asymptotic \fdp\ of the original BH procedure is $\pi_0 \alpha$.

  \item We have: 
    \begin{align}
      \label{eq:critical-values}
      \bs \leq  \asPI  \leq \as_{\textnormal{Sto}(\lambda)} \leq \as_{BH}
    \end{align}
    In models where \eqref{cond:critic} is not satisfied, all the critical values in \eqref{eq:critical-values} are null, implying that all the corresponding procedures have positive power for any target FDR level.  In models where \eqref{cond:critic} is satisfied, all the critical values in \eqref{eq:critical-values} are positive, and \eqref{eq:critical-values} implies that the range of target FDR
    values $\alpha$ that yield asymptotically positive power is larger
    for the plug-in procedures studied in this paper than for the
    BH procedure or the Storey-$\lambda$ procedure.

  \item We have $\tPIinf(\alpha)  \geq \tStoinf(\alpha) \geq 
    \tBHinf(\alpha)$, where $\tStoinf(\alpha)$ denotes the asymptotic threshold of the Storey-$\lambda$ procedure, which is formally defined and characterized in Appendix~\ref{sec:extens-uncond-setting}.  Therefore,  as the power of a thresholding-based FDR
    controlling procedure is a non-decreasing function of its
    threshold (Remark~\ref{rk:increased-power}),
    the asymptotic power
    of the BH$(\alpha/\pih)$ procedure is greater than that of both the Storey-$\lambda$ and the
    original BH procedures, even in the range $\alpha>\as_{BH}$
    where all of them have positive asymptotic power.
  \end{itemize}
\begin{table}[h]
  \centering
  \begin{tabular}{lcccc}
    Name & $\pih$ % & Critical value $/\as$
    & $\fdr/\alpha$ & Rate & (Asy. var. of FDP)/ FDR\\ \hline 

    BH & 
    $1$
    & $\pi_0$ 
    % & $\pi_0 $ 
    & $m^{-1/2}$ & $(\pi_0\tBHinf(\alpha))^{-1}-1$\\

    Oracle BH  &
    $\pi_0$
    & $1$
    % & 1 
    & $m^{-1/2}$ & $(\tBHinf(\alpha/\pi_0))^{-1}-1$\\ 

    Storey-$\lambda$ &
    $\pihSto(\lambda)$ 
    & $\pi_0/\pibSto$ 
    % & $\pi_0/\pibSto$ 
    & $m^{-1/2}$ & $ (\pi_0\tStoinf(\alpha))^{-1} + (1-\pcdf(\lambda))^{-1} $\\

    Kernel($h_m(k)$) &
    $\pih^k$
    & $\pi_0/\ppdf(1)$
    % & $\pi_0/\ppdf(1)$ 
    & $m^{-k/(2k+1)}$ & $\ppdf(1)^{-1}$\\
  \end{tabular}
  \caption{Summary of the asymptotic properties of the FDR controlling procedures considered in this paper, for a target FDR level $\alpha$ greater than the (procedure-specific) critical value.  Note that  ``Storey-$\lambda$'' denotes the original procedure with a fixed $\lambda$, while our extension with $\lambda=1-h_m(k)$ is categorized in the table as a particular case of kernel estimator (last row).  For Storey-$\lambda$, we also assume that $\lambda>\tStoinf(\alpha)$.}
  \label{tab:asymptotic-variances}
\end{table}

  These results characterize the increase in asymptotic power achieved
  by plug-in procedures based on kernel estimators of
  $\pi_0$.  However, this increased asymptotic power comes at the
  price of a slower convergence rate. Specifically, the convergence
  rate of plug-in procedures is the non-parametric rate
  $m^{-k/(2k+1)}/\eta_m$ (where $k$ controls the regularity of
  $\ppdf$) for the BH$(\alpha/\pih^k)$ procedure, while the parametric
  rate $m^{-1/2}$ was achieved by the original BH procedure, the
  Oracle BH procedure, and the Storey-$\lambda$ procedure
  (as proved in Appendix~\ref{sec:extens-uncond-setting}).

\section{Application to location and Student models}
\label{sec:exampl-locat-models}
In Section~\ref{sec:conv-rate-cons} we proved that the asymptotic behavior of plug-in procedures depends on whether the target FDR level $\alpha$ is above or below the critical value $\asPI$ characterized by Theorem~\ref{thm:asy-prop-plug-in}, and by establishing convergence rates for these procedures when $\alpha > \asPI$.  Both the critical value $\asPI$ and the obtained convergence rates depend on the test statistics distribution. In the present section,  these results are applied to Gaussian and Laplace location models, and to  the Student model. 
We begin by defining these models (Section~\ref{sec:models-test-stat})  and studying criticality in each of them (Section~\ref{sec:criticality}). 
Then, we derive convergence rates for plug-in procedures based on the  kernel estimators of $\pi_0$ considered in Sections \ref{sec:conv-rates-estim} and \ref{sec:conv-rate-cons}, both for two-sided tests (Section~\ref{sec:two-sided-models}) and one-sided tests (Section~\ref{sec:one-sided-models}).

\subsection{Models for the test statistics}
\label{sec:models-test-stat}

\paragraph{Location models.} 
In location models the distribution of the test statistic under $\ha$ is a shift from that of the test statistic under $\hn$: $\tcdfa=\tcdfn(\cdot-\theta)$ for some location parameter $\theta>0$. The most widely studied location models are the Gaussian and Laplace (double exponential) location models.  Both the Gaussian and the Laplace distribution can be viewed as instances of a more general class of distributions introduced by~\cite{subbotin23on-the-law-of-frequency} and given for $\gamma \geq 1$ by 
\begin{equation}
  \tpdfn^\gamma(x) = \frac{1}{C_\gamma}\e{-\vert x \vert ^\gamma/\gamma}
, \mbox{ with } C_\gamma= \int_{-\infty}^{+\infty} e^{-|x|^\gamma/\gamma} dx = 2 \Gamma(1/\gamma) \gamma^{1/\gamma-1}\,.
\end{equation}
Therefore, the likelihood ratio in the $\gamma$-Subbotin location model may be written as 
\begin{eqnarray}
\label{eq:lr-subbotin}
  \lrg(x) & = & \expo{\frac{\vert x \vert ^\gamma}{\gamma} - \frac{\left\vert x-\theta \right\vert ^\gamma}{\gamma}}\,.
\end{eqnarray}
The Gaussian case corresponds to $\gamma=2$ and the Laplace case to $\gamma=1$. In the Laplace case, the  distribution of the $p$-values under the alternative can be derived explicitly, see Lemma~\ref{lm:laplace} in Appendix.  We focus on $1 \leq \gamma \leq 2$ as this corresponds to situations in which~\eqref{cond:concavity} is fulfilled.  Specifically, for one-sided tests,~\eqref{cond:concavity} holds as soon as  $\gamma \geq 1$, because then $\ilrg$ is non-decreasing; for two-sided tests, if additionally  $\gamma \leq 2$, then~\eqref{cond:concavity} holds (as proved in Appendix~\ref{sec:calc-spec-models}, Proposition~\ref{prop:concavity-two-sided-subbotin}). 

\paragraph{Student model.}
\label{sec:student-models}
Student's $t$ distribution is widely used in applications, as it naturally arises when testing equality of means of Gaussian random variables with unknown variance.  
In the Student model with parameter $\nu>0$, $\tcdfn$ is the (central) $t$ distribution with $\nu$ degrees of freedom, and $\tcdfa$ is the non-central $t$ distribution with $\nu$ degrees of freedom and non-centrality parameter  $\theta>0$.  The Student model is not a location model, as $\tcdfa$ cannot be written as a translation of $\tcdfn$. 
Following~\citet[Equation (3.5)]{chi07on-the-performance}, we note that the likelihood ratio of the Student model may be written as
  \begin{eqnarray}
    \label{eq:lr-student}
      \lr(t) & = & \sum_{j=0}^{+\infty} a_j(\nu, \theta) \psi_{(j,\nu)}(t)\,,
  \end{eqnarray}
where  $\psi_{(j,\nu)}(t) = (t/\sqrt{t^2+\nu})^j = \sgn(t)^j \left(1+\nu/t^2\right)^{-j/2}$ for $t \in \mathbb{R}$ and 
\begin{align}
\label{eq:def-lr-student-a}
  a_j(\nu, \theta) & =    \e{-\theta^2/2}\frac{\Gamma((\nu+j+1)/2)}{\Gamma((\nu+1)/2)}\frac{(\sqrt{2}\theta)^{j}}{j!}\,.
\end{align}

\begin{remark} The sequence $a_j(\nu, \theta)$ is positive, and it is not hard to see that $(\sum_j a_j(\nu, \theta))$ is a convergent series using Stirling's formula.  Therefore, as  $\psi_{(j,\nu)}(t) \in [-1,1]$, the dominated convergence theorem ensures that Equation~\eqref{eq:lr-student} is well-defined for any $t \in \mathbb{R}$.
\end{remark}

Another useful expression for the Student likelihood ratio may be derived from the integral expression of the density of a non-central $t$ distribution given by \cite{johnson40applications}:
  \begin{align}\label{eq:lr-student-integral}
    \lr(t) & =    
    \exp{\left[-\frac{\theta^2}{2}~\frac{1}{1+\frac{t^2}{\nu}}\right]} ~ 
    \frac{\Hh_\nu\left(-\frac{\theta t}{\sqrt{\nu+t^2}}\right)}{\Hh_\nu(0)}\,,
  \end{align}
where $  \Hh_\nu(z) = \int_0^{+\infty}\frac{u^\nu}{\nu!}e^{-\frac{1}{2}(u+z)^2}dx$.
As noted by \citet[Section 3.1]{chi07on-the-performance}, the likelihood ratio of Student test statistics is non-decreasing, which implies that~\eqref{cond:concavity} holds for one-sided tests. It also holds for two-sided tests, as proved in  Appendix~\ref{sec:calc-spec-models}, Proposition~\ref{prop:concavity-two-sided-student}.\\

The location models and the Student model considered here are parametrized by two parameters: (i) a non-centrality parameter $\theta$, which encodes a notion of distance between $\hn$ and $\ha$; (ii) a parameter which controls the (common) tails of the distribution under $\hn$ and $\ha$: $\gamma$ for the $\gamma$-Subbotin model, and $\nu$ for the Student model with $\nu$ degrees of freedom.

\subsection{Criticality}\label{sec:criticality}
As the asymptotic behavior of plug-in procedures crucially depends on whether the target FDR level is above or below the critical value $\asPI$ characterized by Theorem~\ref{thm:asy-prop-plug-in}, it is of primary importance to study criticality in the models we are interested in.  Noting that $\asPI=\pi_{0, \infty} \as_{BH} = \pi_{0, \infty} \bs/\pi_0$, we have $\asPI>0$ if and only if \eqref{cond:critic} is satisfied, that is, if and only if the likelihood ratio $\ilr$ is bounded near $+\infty$.    In this section, we study \eqref{cond:critic}  in location and Student models. 

\paragraph{Location models.} In location models, where $\tpdfa=\tpdfn(\cdot- \theta)$ with $\theta>0$, the behavior of the likelihood ratio is closely related to the tail behavior of the distribution of the test statistics: for a given non-centrality parameter $\theta$, the heavier the tails, the smaller the difference between $\tpdfa$ and $\tpdfn$.  In a  $\gamma$-Subbotin location model, Equation~\eqref{eq:lr-subbotin} yields $\left\vert 1-\theta/x \right\vert ^\gamma \sim 1-\gamma\theta/x$ as $x \to +\infty$.  Thus $\vert x \vert ^\gamma \left(1- \left\vert 1-\theta/x \right\vert ^\gamma \right) \sim \gamma \theta x^{\gamma-1}$, and the behavior of the likelihood ratio $\ilrg$ is driven by the value of $\gamma$, as illustrated by Figure~\ref{fig:critic-subbotin} for the Gaussian and Laplace location models with location parameter $\theta \in \{1, 2\}$.

\begin{figure}[!ht]
  \centering
  \begin{tabular}{cc}
    Gaussian distribution ($\theta=1$) &     Laplace distribution ($\theta=1$)  \\
    \includegraphics[width=.45\columnwidth, trim=0 30 0 31, clip]{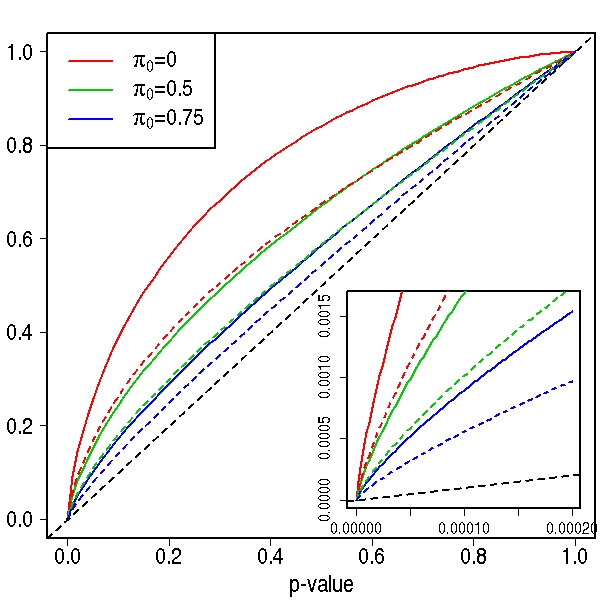} & 
    \includegraphics[width=.45\columnwidth, trim=0 30 0 31, clip]{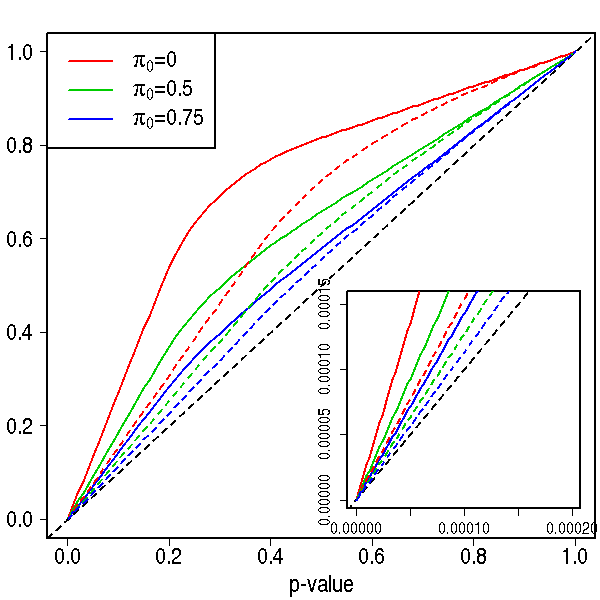}\\
    Gaussian distribution ($\theta=2$) &     Laplace distribution ($\theta=2$)  \\
    \includegraphics[width=.45\columnwidth, trim=0 0 0 31, clip]{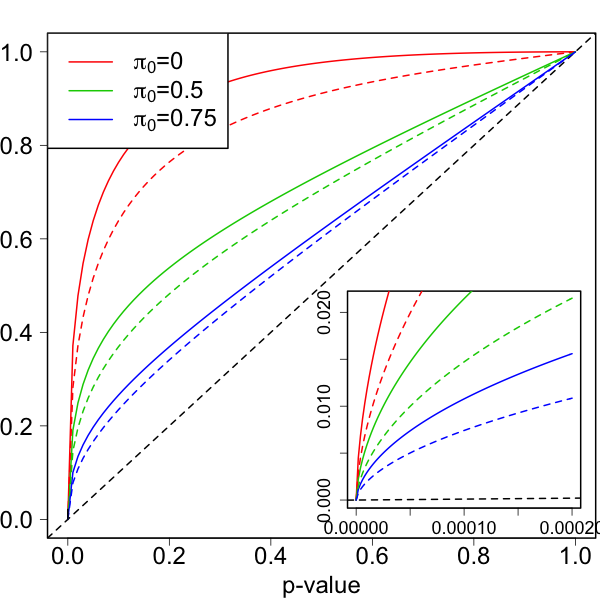} & 
    \includegraphics[width=.45\columnwidth, trim=0 0 0 31, clip]{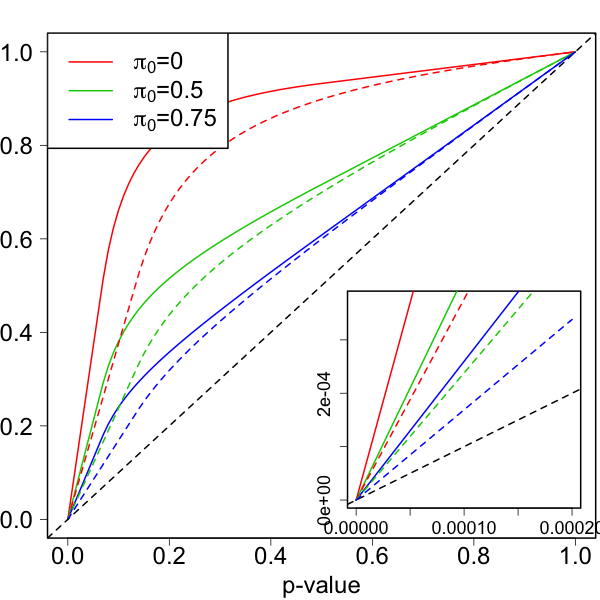}\\
\end{tabular}
\caption{\emph{Distribution functions $\pcdf$ for one-sided (solid) and two-sided (dashed) $p$-values, in Gaussian location models (left: \eqref{cond:critic} is not satisfied), and Laplace location models (right: \eqref{cond:critic} is satisfied) for $\pi_0=$0, 0.5 and 0.75. The location parameter $\theta$ is set to 1 in top panels and 2 in bottom panels.  Inserted plot: zoom in the region $p<2.10^{-4}$.}}
\label{fig:critic-subbotin}
\end{figure}

If $\gamma>1$, then $\lim_{+\infty} \ilrg = +\infty$.  Therefore, the slope of the cumulative distribution function of the $p$-values is infinite at 0, and \eqref{cond:critic} is not satisfied for the Subbotin model: $\bs=0$ for any $\theta$ and $\pi_0$.
This situation is illustrated by Figure \ref{fig:critic-subbotin} (left panels) for the Gaussian model ($\gamma=2$).  In such a situation, for any target FDR level $\alpha$, the asymptotic fraction of rejections by the BH$(\alpha)$ procedure or by a plug-in procedure of the form BH$(\alpha/\pih)$, where $\pih \to \pi_{0,\infty}$ in probability as $m \to +\infty$, is positive by Lemma~\ref{lm:critic-plug-in}. 

If $\gamma = 1$ (Laplace model, as illustrated by Figure~\ref{fig:critic-subbotin}, right panels),  then the likelihood ratio of the model is $\ilrg(x) = \exp(\vert x \vert - \vert x-\theta \vert)$.  It is bounded as $x \to +\infty$, with $\lim_{x \to +\infty} \ilrg(x) = e^\theta$.  Therefore, \eqref{cond:critic} is satisfied for the Laplace location model.  Specifically, we have $\bs = \pi_0/(\pi_0 +(1-\pi_0)\ppdfa(0))$, with $\ppdfa(0) = e^\theta$ for one-sided $p$-values, and $\ppdfa(0) = \cosh{\theta}$ for two-sided $p$-values.   Laplace-distributed test statistics appear as a limit situation in terms of criticality: within the family of $\gamma$-Subbotin location models with $\gamma \in [1,2]$, the Laplace model ($\gamma=1$) is the only  one for which \eqref{cond:critic} is satisfied.

  \paragraph{Student
    model.} 
  For the Student model, Equation~\eqref{eq:lr-student-integral}
  yields that $(\ilr)(t)$ converges to $s_\nu(\theta)$ as $t \to
  +\infty$ and $s_\nu(-\theta)$ as $t \to -\infty$, where $
  s_\nu(\theta) = \Hh_\nu(-\theta)/\Hh_\nu(0)$ is positive for any
  $\theta$.
  Therefore, \eqref{cond:critic} is satisfied for one-sided and two-sided
  tests in the Student model (this had already been noted
  by~\cite{chi07on-the-performance} for one-sided tests).
  Figure~\ref{fig:critic-student} gives the distribution function of
  one- and two-sided p-values in the Student model with parameters
  $\theta \in \{1, 2\}$ and $\nu \in \{10, 50\}$, for $\pi_0 \in \{0,
  0.5, 0.75\}$.  Although criticality is much less obvious than for
  the Laplace model, the inserted plots which zoom into a region where
  the $p$-values are very small ($p < 2.10^{-4}$) do suggest for
  $\nu=10$ that the slope of the distribution function at 0 is linear
  for the Student model.  As an illustration, we calculated that the
  critical values for one-sided tests in the Student model for
  $\pi_0=0.75$ for $\theta \in \{1, 2\}$ are respectively 0.173 and
  0.015 for $\nu=10$, and $4.10^{-3}$ and $7.10^{-6}$ for $\nu=50$.

  \begin{figure}[!ht]
    \centering
    \begin{tabular}{cc}
      Student ($\nu=50$, $\theta=1$) &     Student ($\nu=10$, $\theta=1$)  \\
      \includegraphics[width=.45\columnwidth, trim=0 30 0 31, clip]{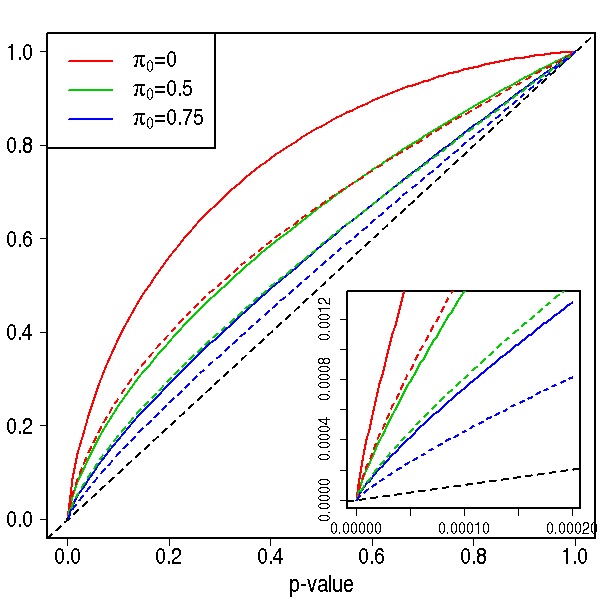} & 
      \includegraphics[width=.45\columnwidth, trim=0 30 0 31, clip]{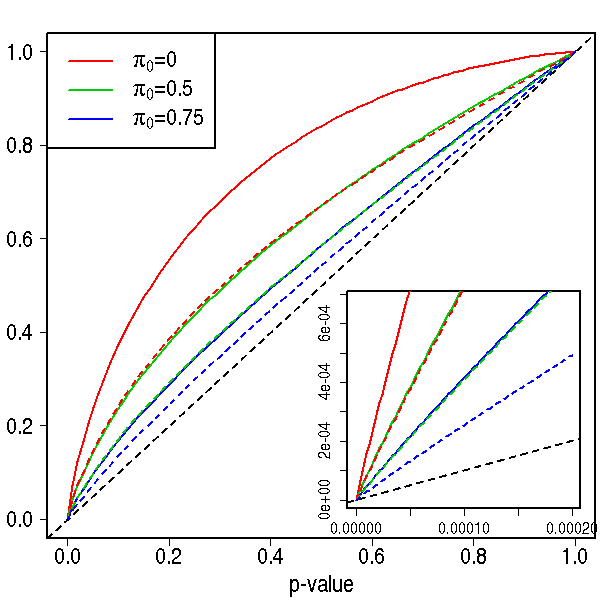} \\
      Student ($\nu=50$, $\theta=2$) &     Student ($\nu=10$, $\theta=2$)  \\
      \includegraphics[width=.45\columnwidth, trim=0 0 0 31, clip]{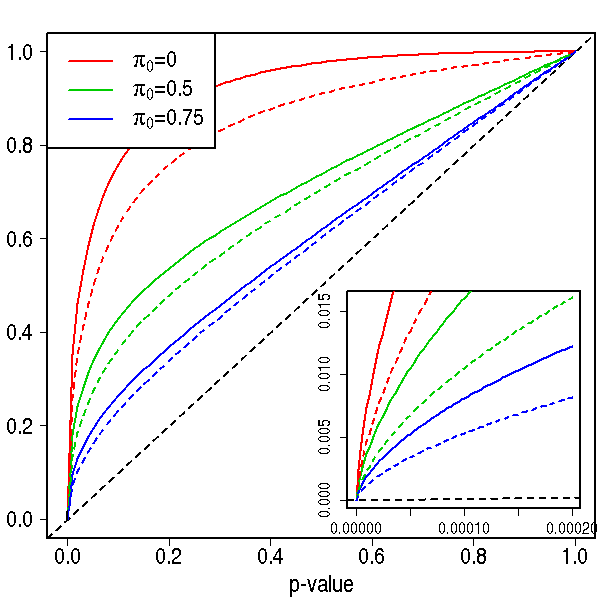} & 
      \includegraphics[width=.45\columnwidth, trim=0 0 0 31, clip]{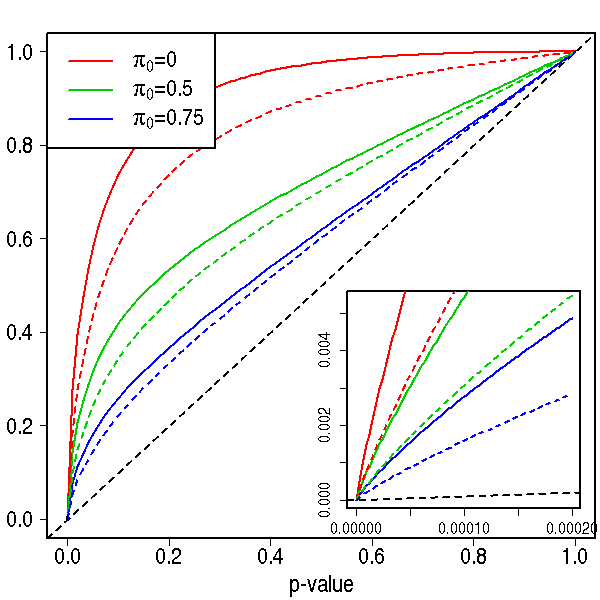}\\
    \end{tabular}
 \caption{\emph{Distribution functions $\pcdf$ for one-sided tests (solid) and two-sided tests (dashed) in Student models with $\nu=100$ degrees of freedom (left) and $\nu=10$ (right). The location parameter $\theta$ was set to 1 in top panels and 2 in bottom panels.  Any Student model satisfies \eqref{cond:critic}.  Inserted plots: zoom in the region $p < 2.10^{-4}$.}}
 \label{fig:critic-student}
\end{figure}
\subsection{Consistency and convergence rates for two-sided tests}
\label{sec:two-sided-models}

\subsubsection*{Consistency.}
Let us first recall that by Proposition~\ref{prop:p-values}.2, we have for two-sided tests under a model satisfying~\eqref{cond:symmetry}:
\begin{equation}\label{eq:ppdfa-two-sided}
  \ppdfats(t) = \frac{1}{2}\left(\lr\left(q_0(t/2)\right) + \lr\left(-q_0(t/2)\right) \right)\,,
\end{equation}
where $q_0: t \mapsto \tcdfn^{-1}(1-t)$ tends to 0 as $t \to 1/2$. A straightforward consequence of (\ref{eq:ppdfa-two-sided}) is that $\ppdfats (1) = (\ilr)(0)$. As $\tpdfa>0$, we have $\ppdf(1) = \pi_0+(1-\pi_0) \ppdfats (1) > \pi_0$.  Therefore,~\eqref{cond:purity} is not met, and the kernel estimators of $\pi_0$ studied in Section~\ref{sec:conv-rates-estim} are not consistent for the estimation of $\pi_0$.  Specifically, we have $\ppdfats(1) =e^{-\theta^2/2}$ for Gaussian and Student test statistics, and $\ppdfats(1) =\e{-\theta}$ for Laplace test statistics.

\subsubsection*{Convergence rates.}
Another consequence of (\ref{eq:ppdfa-two-sided}) is that if for $k\geq 1$ the likelihood ratio $\ilr$ is $k$ times semi-differentiable at 0, then  $\ppdfts$ is $k$ times (left-)differentiable at 1.  
In particular, this holds \emph{for any $k$} in the $\gamma$-Subbotin location model with $\gamma \in [1,2]$, which covers the Gaussian and Laplace cases.  It also holds for the Student model (as proved in Proposition~\ref{prop:lr-student-derivatives}).  For these models, Corollary~\ref{cor:fdp-opt-bw} entails that for any $k>0$, if $\pih$ is a kernel estimator of $\ppdf$ associated with a $k^{\rm th}$ order kernel with bandwidth $h_m(k)=m^{-1/(2k+1)}\eta_m^2$ (where $\eta_m \to 0$ and $m \eta_m \to +\infty$ as $m \to +\infty$), then the corresponding plug-in procedure BH$(\alpha/\pih)$ converges in distribution at rate $m^{-k/(2k+1)}/\eta_m$ for any $\alpha$ greater than $\asPI=\ppdf(1) \as_{BH}$. 
These results are summarized in the last column of Table~\ref{tab:location-models}.

Let us now consider the modification of the Storey-$\lambda$ estimator introduced in Section~\ref{sec:conv-rates-estim}: $\pih=\pihSto(1-h_m)$, with $h_m \to 0$ as $m \to +\infty$.  By Corollary~\ref{cor:fdp-opt-bw}, the optimal convergence rate of the  BH$(\alpha/\pih)$ procedure is then determined by the order of the first non null derivative of $\ppdfts$ at 1.  In order to calculate this order, we use the following lemma:

\begin{lemma}[Behavior of $\ppdfa$ at 1 for two-sided $p$-values in symmetric models]
\label{lm:regularity-two-sided-p}
Under~\eqref{cond:symmetry}, the density function  $\ppdfats$  of two-sided $p$-values under the alternative hypothesis satisfies:
  \begin{enumerate}
  \item If $\ilr$ is semi-differentiable at 0, with left-derivative $\ell_-$ and right-derivative $\ell_+$, then $\ppdfats^{(1)}$ is semi-differentiable at 1 and we have:
    \begin{displaymath}
      \ppdfats^{(1)} (1) = - \frac{\ell_+-\ell_-}{4\tpdfn(0)}\,.
    \end{displaymath}
In particular, $\ppdfats^{(1)} (1) =0$ if and only if $\ilr$ is differentiable at 0.
  \item If $\ilr$ is twice differentiable at 0, then $ \ppdfats^{(1)}$ is twice differentiable at 1 and we have:
    \begin{displaymath}
      \ppdfats^{(2)} (1) = \frac{1}{4\tpdfn(0)^2}\left(\lr\right)^{(2)}(0)\,.
    \end{displaymath}
  \end{enumerate}
\end{lemma}
Lemma~\ref{lm:regularity-two-sided-p} may be applied to  two-sided tests for $\gamma$-Subbotin location models, and for the Student model.  For the two-sided Gaussian model, $\ilr$ is $C^\infty$ near 0 and $(\ilr)^{(2)}(0) \neq 0$.  The same holds for the  two-sided Student model, as shown in Appendix~\ref{sec:student-model} (Proposition~\ref{prop:lr-student-derivatives}). For both models, Lemma~\ref{lm:regularity-two-sided-p} entails that  $\ppdfts^{(1)}(1)=0$ and $\ppdfts^{(2)}(1)>0$.  For two-sided Laplace test statistics, the likelihood ratio $\ilr:t \mapsto \expo{\vert t-\theta \vert - \vert t \vert}$ has a singularity at $t=0$ but it is semi-differentiable at 0 (and differentiable on $(-\infty, \theta)\setminus \{0\}$), with left and right derivatives at 0 given by $\ell_- = 0$ and $\ell_+ = \e{-\theta}$. Lemma~\ref{lm:regularity-two-sided-p} yields that %$\ppdfts$ is differentiable at 1, with 
 $\ppdfts^{(1)}(1)=-(1-\pi_0)\e{-\theta}/2$. 
In particular, letting $k=1$ for the Laplace model and $k=2$ for the Gaussian and Student models, Corollary~\ref{cor:fdp-opt-bw} yields that if $\pih = \pihSto(1-m^{-1/(2k+1)}\eta_m^2)$, where $\eta_m \to 0$, then for any $\alpha > \asPI=\ppdf(1) \as_{BH}$, the FDP of the BH$(\alpha/\pih)$ procedure converges in distribution at rate $m^{-k/(2k+1)}/\eta_m$ toward $\pi_0\alpha/\ppdf(1)$, where $\ppdf(1) = \pi_0 + (1-\pi_0) e^{-\theta^2/2}$ in the Gaussian and Student models, and $\ppdf(1) = \pi_0 + (1-\pi_0) e^{-\theta}/2$ in the Laplace model. 
These rates are slower than those obtained at the beginning of this section for $k^{\mathrm{th}}$ order kernels because the latter do not require the derivatives of $\ppdf$ of order $l<k$ to be null at 1, which implied that any $k>0$ could be chosen (see Table~\ref{tab:location-models} for a comparison).

\subsection{Consistency and convergence rates for one-sided tests}
\label{sec:one-sided-models}

\subsubsection*{Consistency.}
For one-sided tests, we have $\ppdfaos(t) = (\ilr)(q_0(t))$.  As  $\lim_{t \to 1} q_0(t)=-\infty$,~\eqref{cond:purity} is met if and only if the likelihood ratio $(\ilr)(t)$ tends to 0 as $t \to-\infty$.  For the Student model, $\ilr$ tends to $s_\nu(-\theta)>0$ as $t \to -\infty$.  This implies that~\eqref{cond:purity} is not satisfied in that model: $\pi_0$ cannot be consistently estimated using a consistent estimator of $\ppdfos(1)$, because $\ppdfos(1) =\pi_0 +(1-\pi_0)e^{-\theta}> \pi_0$. 
  For location models, we begin by establishing a connection between purity and criticality (Proposition~\ref{prop:purity-critic}), which is a consequence of the following symmetry property:
\begin{lemma}[Likelihood ratios in symmetric location models]
  \label{lm:purity-critic}
  Consider a location model
in which the test statistics have densities $\tpdfn$ under $\hn$, and $\tpdfa=\tpdfn(\cdot-\theta)$ under $\ha$ for some $\theta \neq 0$. Under \eqref{cond:symmetry}, we have
  \begin{displaymath}
    \lim_{-\infty} \frac{\tpdfn}{\tpdfa} = \lim_{+\infty} \frac{\tpdfa}{\tpdfn}\,.
  \end{displaymath}
\end{lemma}
For one-sided tests in symmetric location models, Lemma~\ref{lm:purity-critic} implies the following result:
\begin{proposition}[Purity and criticality for one-sided tests in symmetric location models]
  \label{prop:purity-critic}
Let $\ppdfaos$ be the density of \emph{one-sided} $p$-values under the alternative hypothesis, and $\bs$ the critical value of the multiple testing problem. Under~\eqref{cond:symmetry} and~\eqref{cond:concavity},
  \begin{enumerate}
  \item \eqref{cond:critic} and~\eqref{cond:purity} are complementary events, in the sense that $\bs=0$ if and only if $\ppdfaos(1)=0$;
  \item If $\lim_{+\infty} \ilr$ is finite, then $\bs = \pi_0/\left(\pi_0 +(1-\pi_0)\ppdfaos(0)\right)$ and $\ppdfos(1)=\pi_0 +(1-\pi_0)\ppdfaos(1)$ are connected by $\ppdfaos(0) \ppdfaos(1)=1$.
  \end{enumerate}
\end{proposition}
Proposition~\ref{prop:purity-critic} implies that contrary to two-sided location models, in which we always have $\ppdfats(1)>0$, consistency \emph{may} be achieved in one-sided location models using kernel estimators such as those considered here, depending on model parameters.  In particular, there is no criticality in the one-sided Gaussian model, implying that~\eqref{cond:purity} is satisfied in that model: we have $\ppdfos(1) = \pi_0$, and $\pi_0$ can be consistently estimated using the kernel estimators of $\ppdfos(1)$ introduced in Section~\ref{sec:conv-rates-estim}.   In the one-sided Laplace model, \eqref{cond:critic} is satisfied, implying that~\eqref{cond:purity} is not satisfied in that model: $\pi_0$ cannot be consistently estimated using these kernel estimators of $\ppdfos(1)$. 

\subsubsection*{Convergence rates.}
  \paragraph{Student.}
  Proposition~\ref{prop:lr-student-derivatives} entails that for the
  one-sided Student model, $\ppdfa$ is $C^\infty$, and all its
  derivatives or order greater than 1 are null at 1.  Therefore, any
  $k>0$, if $\pih^k$ denotes any of the two estimators studied in
  Corollary~\ref{cor:fdp-opt-bw} for a $k^{\rm th}$ order kernel with
  bandwidth $h_m(k)=m^{-1/(2k+1)}\eta_m^2$ (where $\eta_m \to 0$ and
  $m \eta_m \to +\infty$ as $m \to +\infty$), then the corresponding
  plug-in procedure BH$(\alpha/\pih^k)$ converges in distribution at
  rate $m^{-k/(2k+1)}/\eta_m$ for any $\alpha$ greater than
  $\asPI=\ppdf(1) \as_{BH}$.  These results are summarized in the
  first row of Table~\ref{tab:location-models}.

\paragraph{Laplace.} The distribution of one-sided $p$-values in the one-sided Laplace model satisfies $\pcdfaos(t) = 1-(1-t)e^{-\theta}$ for $t \geq 1/2$, see Lemma~\ref{lm:laplace} in Appendix~\ref{sec:calc-spec-models}. Therefore, for $t\geq 1/2$,  $(1-\pcdfos(t))/(1-t)$ is constant, equal to $\ppdfos(1) = \pi_0 + (1-\pi_0)e^{-\theta}$, as illustrated by the solid curves in the right panels of Figure~\ref{fig:critic-subbotin}. Therefore, for any fixed $\lambda \geq 1/2$, the Storey-$\lambda$ estimator is an unbiased estimator of $\ppdf(1)$, which converges to $\ppdf(1)$ at rate $m^{-1/2}$. The same property holds for any kernel estimator of $\ppdf(1)$ with a fixed bandwidth.  
These results are summarized in the third row of Table~\ref{tab:location-models}.

\paragraph{Gauss.} In the Gaussian model however, the regularity of $\ppdfa$ near 1 is poor: we have
\begin{displaymath}
  \ppdfaos(t) = \exp \left(-\frac{\theta^2}{2}-\theta \Phi^{-1}(t)\right)\,,
\end{displaymath}
where $\Phi(=\tcdfn)$ denotes the standard Gaussian distribution function. As $h \to 0$, $\Phi^{-1}(1-h) \leq \sqrt{2\ln(1/h)}$, implying that
\begin{displaymath}
  \ppdfaos(1-h) \geq \exp \left(-\frac{\theta^2}{2}-\theta\sqrt{2\ln(1/h)}\right)\,.
\end{displaymath}
Therefore, $\ppdfaos$ is not differentiable at 1, and the convergence rates of the kernel estimators of $\pi_0$ studied in Section~\ref{sec:conv-rates-estim} are slower than $m^{-1/3}$ in our setting.  
These results are summarized in the second row of Table~\ref{tab:location-models}.

\begin{table}[!htp]
  \centering
  \begin{tabular}{l|c|c|c|c}
    &&& \multicolumn{2}{c}{Convergence rates} \\
    Model & $\lim_0 1/\ppdfa$ & $\ppdfa(1)$ & $\pihSto(1-h_m(k))$ & $\hat{\ppdf}_{m}^{k}(1)/\eta_m$ \\ \hline
    One-sided  Student & $s_\nu(\theta)$ & $s_\nu(-\theta)$ & $\ll m^{-k/(2k+1)}/\eta_m $ & $\ll m^{-k/(2k+1)}/\eta_m$ \\
    One-sided  Gaussian & 0 & 0 & $\gg m^{-1/3}$ & $\gg m^{-1/3}$ \\
    One-sided Laplace & $e^{-\theta}$ & $e^{-\theta}$  & $m^{-1/2}$ & $m^{-1/2}$ \\
    \hline
   Two-sided Student & $(s_\nu(\theta)+s_\nu(-\theta))/2$ & $e^{-\theta^2/2}$ & $m^{-2/5}/\eta_m$ & $\ll m^{-k/(2k+1)}/\eta_m$\\
    Two-sided Gaussian & 0 & $e^{-\theta^2/2}$ & $m^{-2/5}/\eta_m$ & $\ll m^{-k/(2k+1)}/\eta_m$\\
    Two-sided Laplace & $\cosh{\theta}$ & $e^{-\theta}$ & $m^{-1/3}/\eta_m$& $\ll m^{-k/(2k+1)}/\eta_m$\\
  \end{tabular}
  \caption{\emph{Properties of one- and two-sided test statistics distributions in Student, Gaussian, and Laplace models, and convergence rates of the kernel estimators studied. 
When the rate depends on $k$, the value of $k$ may be chosen arbitrarily large. 
$\eta_m$ is a sequence such that $\eta_m \to 0$ and $m\eta_m \to +\infty $ as $m \to +\infty$.}} 
  \label{tab:location-models}
\end{table}

\begin{figure}[!h]
  \centering
  \includegraphics[width=\columnwidth]{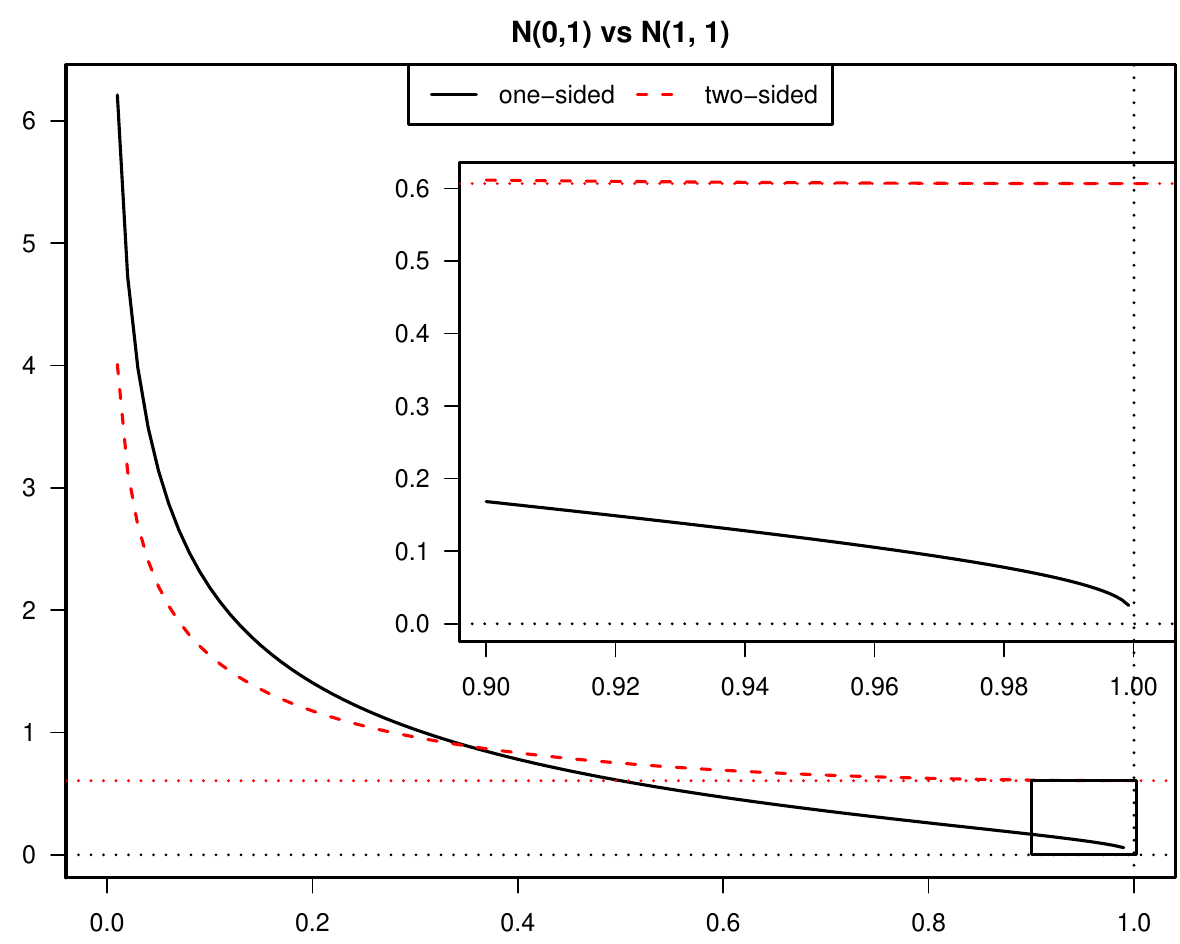}
  \caption{\emph{Density of one- and two-sided $p$-values under the alternative hypothesis for the location model $\mathcal{N}(0,1)$ versus $\mathcal{N}(1,1)$. Inserted plot: zoom in the region $[0.9, 1]$, which is highlighted by a black box in the main plot.}}
  \label{fig:gaussian}
\end{figure}

The difference between one- and two-sided tests in the Gaussian location model is illustrated by Figure \ref{fig:gaussian} for $\theta=1$, that is when testing $\mathcal{N}(0,1)$ against $\mathcal{N}(1,1)$.
The density of two-sided $p$-values has a positive limit at 1, and its derivative at 1 is 0, making it possible to estimate $\ppdf(1)=\pi_0+(1-\pi_0)e^{-\theta^2/2}$ at rate $m^{-2/5}$, by Corollary~\ref{cor:fdp-opt-bw}. Conversely, the density of one-sided $p$-values tends to 0 at 1, but is not differentiable: the true $\pi_0$ can be estimated consistently, but the convergence rate is slower.

\section{Concluding remarks}
\label{sec:discussion}
  This paper studies asymptotic properties of a family of plug-in
  procedures based on the BH procedure.  When compared to the BH procedure or to the Storey-$\lambda$ procedure, the results for general
  models obtained in Section~\ref{sec:conv-rate-cons} show that
  incorporating the proposed estimators of $\pi_0$ into the BH procedure 
  asymptotically yields (i) tighter FDR control (or, equivalently, greater power) and (ii) smaller
  critical values, thereby increasing the range of situations in which
  the resulting procedure has positive asymptotic power.  These improvements come at the price of a reduction in the
  convergence rate from the parametric rate $m^{-1/2}$ to a
  non-parametric rate $m^{-k/(2k+1)}$, where $k$ is connected to the order of differentiability of the test statistics distribution.  As the results obtained
  for the proposed modification of the Storey-$\lambda$ estimator $\pihSto(1-h_m)$ require stronger conditions (null derivatives
  of $\ppdfa$) than for kernel estimators with a kernel of order $k$,
  we conclude that it is generally better to use the latter class of
  estimators.

  Our application of these results to specific models for the test statistics sheds some light on the influence of the test statistics distribution on  convergence rates of plug-in procedures: 
  \begin{itemize}
  \item When the test statistics distribution is $C^\infty$ (e.g. for two-sided Gaussian test statistics, and for Laplace and Student tests statistics), the obtained convergence rates  are slower than the parametric rate, but may be arbitrarily close to it by choosing a kernel of sufficiently high order.
    The resulting estimators are not consistent estimators of $\pi_0$, although the bias decreases as the non-centrality parameter $\theta$ increases.
  \item When the regularity of the test statistics distribution is poor (such as in the one-sided Gaussian model), the convergence rate of the FDP achieved by the plug-in procedures studied in this paper is slower.  The plug-in procedures studied  are still asymptotically more powerful than the BH procedure or the Storey-$\lambda$ procedure, but the FDP actually achieved by that procedure may be far from the target FDR level.
  \end{itemize}
  Obtaining
    more precise conclusions in the context of a specific data set
    or application exceeds the scope of the present paper, as it  would require extending the obtained results to more realistic settings such as the ones that are now described.

\subsection{Extensions of the multiple testing setting considered}
\label{sec:extens-mult-test}
An interesting research direction would be to extend the multiple testing setting considered here to more realistic assumptions.  A typical example of application is the case of differential expression analyses in genomics, which aim at identifying those genes whose expression level differs between two known populations of samples.  
First, we have assumed that all null hypotheses are independent, and that all true alternative hypotheses follow the same distribution. The independence assumption is not realistic, as genes are known to interact with each other, in particular through transcriptional regulation networks.  Moreover, the level of differential expression needs not be the same for all genes under $\ha$.
For the results on criticality that have been used in this paper, the proof given in \cite{chi07on-the-performance} essentially relies on the assumption that the $p$-values are independently and identically distributed.  Therefore, it seems that these results could be extended to composite distributions under $\ha$, provided that the corresponding marginal distributions are still independently and identically distributed.  Extending these results to settings where the independence assumption is relaxed seems a more challenging question.  As for the convergence results established in Section~\ref{sec:conv-rate-cons}, their proofs rely on the formalism laid down by  \cite{neuvial08asymptotic}.  Therefore, these results could be extended  to other dependency assumptions, or to composite distributions under $\ha$ provided that the convergence in distribution of the empirical distribution functions $(\pecdfn, \pecdfa)$ holds under these assumptions.  In that spirit, the results of \cite{neuvial08asymptotic} have recently been extended to an equi-correlated Gaussian model~\citep{delattre11on-the-false} and to a more general Gaussian model where the covariance matrix is supposed to be close enough to the identity as the number of tests grows to infinity~\citep{delattre:asymptotics}.

Second, we have shown that the asymptotic properties of FDR controlling procedures are driven by the shape and regularity of the test statistics distribution.  In practice, the test statistics distribution depends on the size of the sample used to generate them.  In differential expression analyses, a natural test statistic is Student's $t$, whose distribution depends on sample size through both the number of degrees of freedom $\nu$  and a non-centrality parameter $\theta$. In the spirit of the results of \cite{chi07sample} on the influence of sample size on criticality, it would be interesting to study the convergence rates of plug-in procedures when both the sample size and the number of hypotheses tested grow to infinity.

\subsection{Alternative strategies to estimate $\pi_0$}
\label{sec:other-strat-estim}
The estimators of $\pi_0$ considered in this paper are kernel estimators of the density $\ppdf$ at 1.  Therefore, they achieve non-parametric convergence rates of the form $m^{-k/(2k+1)}/\eta_m$, where $k$ controls the regularity of $\ppdf$ near 1 and $\eta_m \to 0$ slowly enough.  An interesting open question is whether these non-parametric rates may be improved.  Other strategies for estimating $\pi_0$  may be considered to achieve faster convergence rates, including the following two:
\begin{itemize}
\item One-stage adaptive procedures as proposed by \cite{blanchard09adaptive} and \cite{finner09on-the-false} allow more powerful \fdr\ control than the standard BH procedure without explicitly incorporating an estimate of $\pi_0$: they are not plug-in procedures. 
\item \cite{jin08proportion} proposed an estimator of $\pi_0$ based on the Fourier transform of the empirical characteristic function of the $Z$-scores associated to the $p$-values.  This estimator does not focus on the behavior of the density near 1, and might not suffer from the same limitations as the estimators studied here. This estimator was shown to be consistent for the estimation of $\pi_0$ when the $Z$-scores follow a Gaussian location mixture, but no convergence rates were established.
\end{itemize}
  In a general semi-parametric framework where $\ppdfa$ is not
  necessarily decreasing, and its regularity is not specified,
  \cite{nguyen13efficient} have recently proved that if the Lebesgue
  measure of the set on which $\ppdfa$ achieves its
  minimum is null, then no consistent estimator of $\min_t g(t)$
  with a finite asymptotic variance can reach the parametric
  convergence rate $m^{-1/2}$.  In our setting where $\ppdfa$ is
  decreasing, the measure of the set on which $\ppdfa$ is minimum is
  indeed null, except if $\ppdfa$ is constant on an interval of the
  form $[t_0, 1]$.  For one-sided tests where $\ppdfa(t) = (\ilr)
  (\tcdfn^{-1}(1-t))$, this extreme situation arises if and only if
  the likelihood ratio is constant on an interval of the form $[x_0,
  +\infty)$.  Among all models studied in
  Section~\ref{sec:exampl-locat-models}, the only case in which this
  occurs is the one-sided Laplace model, where $\ilr(x) = \exp(\vert x
  \vert - \vert x-\theta \vert) = e^{\theta}$ for $x \geq \theta > 0$.
  The kernel estimators that we have studied here do reach the rate
  $m^{-1/2}$ in this case.

In the more common situation in which the measure of the set on which $\ppdfa$ vanishes (or achieves its minimum) is null, the above negative result of \cite{nguyen13efficient} suggests that there is little room for improving on the non-parametric convergence rates obtained in Propositions~\ref{prop:opt-bw-sto}  and~\ref{prop:opt-bw-k}. 
We conjecture that it is not possible for consistent estimators of $g(1)$ to reach a parametric convergence rate in this setting. 

\acks{
The author would like to thank Stéphane Boucheron for insightful advice and comments, and Catherine Matias and Etienne Roquain for useful discussions.
He is also grateful to anonymous referees for very constructive comments and suggestions that greatly helped improve the paper.

This work was partly supported by the association ``Courir pour la vie, courir pour Curie'', and by the French ANR project TAMIS.
}

\appendix                  
\section*{Appendix}
\addcontentsline{toc}{section}{Proofs}

\section{Calculations in specific models} 
\label{sec:calc-spec-models}

\subsection{Location models}
\label{sec:location-models}

Lemma~\ref{lm:laplace} gives the distribution of the $p$-value under the alternative hypothesis for one-sided tests in the Laplace model.  The proof is straightforward, so it is omitted.

\begin{lemma}[One-sided Laplace location model] \label{lm:laplace}
Assume that the probability distribution function of the test statistics is $\tpdfn:x \mapsto \frac{1}{2} \e{-|x|}$ under the null hypothesis, and $\tpdfa:x \mapsto \frac{1}{2} \e{-|x-\theta|}$ under the alternative, with $\theta>0$ (one-sided test). Then
\begin{enumerate}
\item The one-sided \emph{p}-value function is
  \begin{eqnarray*} 
    1-\tcdfn(x)=
   \begin{cases}
      \frac{1}{2} \e{(-|x|)} & \textrm{ if } x \geq 0\\
      1-\frac{1}{2} \e{(-|x|)} & \textrm{ if } x<0\\
   \end{cases}
  \end{eqnarray*}
\item The inverse one-sided \emph{p}-value function is
  \begin{eqnarray*} 
    \left(1-\tcdfn\right)^{-1}(t)=
   \begin{cases}
      \ln{\left(\frac{1}{2t}\right)} & \textrm{ if } 0 \leq t \leq \frac{1}{2}\\
      \ln{\left(2(1-t)\right)} & \textrm{ if } \frac{1}{2} < t < 1 \\
   \end{cases}
  \end{eqnarray*}
\item The cdf of one-sided \emph{p}-values under $\ha$ is 
  \begin{eqnarray*} 
    \pcdfaos(t)=
   \begin{cases}
      t \e{\theta} & \textrm{ if } 0 \leq t \leq  \frac{\e{-\theta}}{2}\\
      1-\frac{1}{4t} \e{-\theta}& \textrm{ if } \frac{\e{-\theta}}{2} \leq t \leq \frac{1}{2}\\
      1-(1-t) \e{-\theta} &   \textrm{ if } t \geq \frac{1}{2}\\
   \end{cases}
  \end{eqnarray*}
\item The probability distribution function of one-sided \emph{p}-values under $\ha$ is 
  \begin{eqnarray*}
    \ppdfaos(t)=
   \begin{cases} \label{expo-pdf}
      \e{\theta} & \textrm{ if } 0 \leq t \leq  \frac{\e{-\theta}}{2}\\
      \frac{1}{4t^2} \e{-\theta}& \textrm{ if } \frac{\e{-\theta}}{2} \leq t \leq \frac{1}{2}\\
      \e{-\theta} &   \textrm{ if } t \geq \frac{1}{2}\\
   \end{cases}
  \end{eqnarray*}
\end{enumerate}
\end{lemma}

\begin{proposition}[Concavity in two-sided $\gamma$-Subbotin models]
  \label{prop:concavity-two-sided-subbotin}
If the test statistics follow a $\gamma$-Subbotin distribution with $\gamma \in [1,2]$, then the distribution function of the two-sided $p$-values under the alternative $\pcdfa$ is concave.
\end{proposition}

\begin{proof}[Proof of Proposition~\ref{prop:concavity-two-sided-subbotin}]
\eqref{cond:symmetry} holds for Subbotin models. By Lemma~\ref{lm:concavity}, we need to prove that the likelihood ratio $\ilrg$ of the $\gamma$-Subbotin model with $\gamma$ is such that $h:x \mapsto(\ilrg)(x) + (\ilrg)(-x)$ is non-decreasing on $\mathbb{R}_+$. The function $h$ is differentiable on $(0,+\infty)\setminus \{\theta\}$, and its derivative is given by 
\begin{eqnarray*}
  \label{eq:h-prime}
  h'(x) &=& \left(\lrg \right) '(x)-\left(\lrg\right)'(-x) \,,
\end{eqnarray*}
where 
\begin{eqnarray}
  \label{eq:lrg-prime}
  \left(\lrg\right)'(y) & = & \left(\sgn(y) \vert y \vert ^{\gamma-1} - \sgn(y-\theta) \left\vert y-\theta \right\vert ^{\gamma-1}\right)\lrg(y)
\end{eqnarray}
for any $y \in \mathbb{R}\setminus \{0,\theta\}$.
Let $x>0$ such that $x \neq \theta$, we are going to prove that $h'(x) \geq 0$.
As $\ilrg$ is non-decreasing, both $\left(\ilrg\right)'(x)$ and $\left(\ilrg\right)'(-x)$ are non-negative.
If $\left(\ilrg\right)'(-x)=0$, then $h'(x) \geq 0$ as desired.  From now on, we assume that $\left(\ilrg\right)'(-x) > 0$.
As $\theta>0$, \eqref{eq:lrg-prime} entails that 
\begin{eqnarray}
\label{eq:lr-subbotin}
\frac{\left(\ilrg\right)'(x)}{\left(\ilrg\right)'(-x)} 
& = &  \frac{ x ^{\gamma-1} - \sgn(x-\theta) \left\vert x-\theta \right\vert ^{\gamma-1}}{\left( x + \theta \right) ^{\gamma-1}- x ^{\gamma-1}} \frac{  \tpdfa(x)^\gamma}{\tpdfa(-x)^\gamma} ,
\end{eqnarray}
where $\tpdfa(x)^\gamma > \tpdfa(-x)^\gamma$ because $-\vert x-\theta\vert + \vert x+\theta\vert >0$.  As $\left(\ilrg\right)'(-x) > 0$, it is enough to show that 
\begin{eqnarray}
\label{cond:lr-subbotin}
x ^{\gamma-1} - \sgn(x-\theta) \left\vert x-\theta \right\vert ^{\gamma-1} \geq \left( x+\theta \right) ^{\gamma-1}- x ^{\gamma-1}
\end{eqnarray}
in order to prove that $h'(x) \geq 0$.  
By the  concavity of $x \mapsto x^{\gamma-1}$ on $\mathbb{R}_+$ for $1\leq \gamma \leq 2$,  $\phi:x \mapsto \theta^{-1}(x^{\gamma-1} - (x-\theta)^{\gamma-1})$ is non-increasing  on $[\theta, +\infty]$.  Therefore, if $x > \theta$ we have $\phi(x) \geq \phi(x+\theta)$ and~\eqref{cond:lr-subbotin} holds.
If $x < \theta$, then noting that for any $a,b >0$ and $\zeta \in [0,1]$, $a^\zeta + b^\zeta \geq (a+b)^\zeta$, we have, for $1\leq \gamma \leq 2$, $x^{\gamma-1} + (\theta-x)^{\gamma-1} \geq \theta^{\gamma -1} \geq \left( x+\theta \right) ^{\gamma-1}- x ^{\gamma-1}$, and~\eqref{cond:lr-subbotin} holds as well.
\end{proof}
 
  \subsection{Student model}
  \label{sec:student-model}
  \begin{lemma}[Derivative of the Student likelihood ratio]
    \label{lm:lr-student-derivatives}
    Let $\nu\in \mathbb{N}^*$ and $\theta>0$.  The likelihood ratio
    $\ilr$ of the Student model with $\nu$ degrees of freedom and
    non-centrality parameter $\theta$ is $C^1$ 
    on $\mathbb{R}$, and for any
    $t \in \mathbb{R}$,
    \begin{align}
      \label{eq:lr-student-derivative}
      \left(\lr\right)'(t) =  \nu(\nu + t^2) ^{-3/2}
      \sum_{j=0}^{+\infty} a_j^{1}(\nu, \theta)\psi_{(j,\nu)}(t)\,,
    \end{align}
where $a_j^{1}(\nu, \theta)=(j+1)a_{j+1}(\nu, \theta)$ is such that $(\sum_j a_j^{1}(\nu,\theta))$  converges absolutely.
  \end{lemma}
\begin{proof}[Proof of Lemma~\ref{lm:lr-student-derivatives}]
As $(\sum_j  a_j(\nu,\theta))$ converges absolutely and as $\psi_{(j,\nu)}$ is
  differentiable on $\mathbb{R}$ for any $j \geq 0$ and bounded (by
  [-1,1]), the dominated convergence theorem ensures that $\ilr$ is
  differentiable on $\mathbb{R}$ and that its derivative is given by:
  \begin{align}
    \left(\lr\right)'(t) & = \sum_{j=1}^{+\infty} a_j(\nu, \theta)
    \psi'_{(j,\nu)}(t)\,.
  \end{align}
  For $t \neq 0$, we have $\log \left(\sgn(t)^j
    \psi_{(j,\nu)}(t)\right) = - j/2 \left(\log(1+\nu/t^2)\right)$,
  whose derivative is $j\nu/(\nu t + t^3)$, so that
  \begin{align}
    \psi'_{(j,\nu)}(t) & = \psi_{(j,\nu)}(t) \frac{j\nu}{t(\nu +
      t^2)}\,.
  \end{align}
  As $\psi_{(j,\nu)}(t) \underset{t \to 0}{\sim} (t/\sqrt{\nu})^j$, we have
  $\psi_{(j,\nu)}(0)=0$, $\psi'_{(j,\nu)}(0)=0$, and $\psi'_{(j,\nu)}$
  is continuous at 0.
  Equation~\eqref{eq:lr-student-derivative} follows by noting that $\psi_{(j+1,\nu)}(t)/\psi_{(j,\nu)}(t) =  t/\sqrt{t^2+\nu}$, and that $(\sum_j a_j^{1}(\nu,\theta))$ converges absolutely by Stirling's formula.
\end{proof}

Lemma~\ref{lm:lr-student-derivatives} entails the following result:
\begin{proposition}[Regularity of the Student likelihood ratio]
   \label{prop:lr-student-derivatives}
    Let $\nu\in \mathbb{N}^*$ and $\theta>0$.  The likelihood ratio
    $\ilr$ of the Student model with $\nu$ degrees of freedom and
    non-centrality parameter $\theta$ is has the following properties:
  \begin{enumerate}
  \item $\ilr$ is $C^\infty$ on $\mathbb{R}$;
  \item For any $k \in \mathbb{N}^*$, we
    have 
    $\left(\ilr\right)^{(k)}(t) \to 0$ as $\vert t \vert \to +\infty$;
  \item $(\ilr)^{(2)}(0) \neq 0$.
  \end{enumerate}
\end{proposition}

\begin{proof}[Proof of Proposition~\ref{prop:lr-student-derivatives}]
  \begin{enumerate}
  \item By \eqref{lm:lr-student-derivatives},  the function series in $(\ilr)'$  has the same form as $\ilr$; therefore, the result easily follows by induction.
  \item By \eqref{lm:lr-student-derivatives}, Leibniz formula entails that the successive derivatives of $\ilr$ are linear combinations of products of function series of the same form as $\ilr$ by derivatives of $t \mapsto (\nu + t^2) ^{-3/2}$.  The result follows by the dominated convergence theorem, as all the derivatives of $t \mapsto (\nu + t^2) ^{-3/2}$ tend to 0 as $\vert t \vert \to +\infty$;
  \item The result follows by differentiating \eqref{eq:lr-student-derivative} at 0.
 \end{enumerate}
\end{proof}

\begin{proposition}[Concavity in the two-sided Student model]
  \label{prop:concavity-two-sided-student}
  The distribution function $\pcdfa$ of two-sided $p$-values in the
  Student model satisfies~(\ref{cond:concavity}).
\end{proposition}

\begin{proof}[Proof of Proposition~\ref{prop:concavity-two-sided-student}]
  By Lemma~\ref{lm:concavity}, we need to prove that the likelihood
  ratio $\ilr$ of the Student model is such that $t \mapsto (\ilr)(t)
  + (\ilr)(-t)$ is non-decreasing.
  Equation~\eqref{eq:lr-student-derivative} yields for $t \in
  \mathbb{R}$
  \begin{eqnarray}
    \label{eq:sum-derivative-lr-student}
    \left(\lr\right)'(t) +  \left(\lr\right)'(-t)&  = &  \nu(\nu + t^2) ^{-3/2}~\sum_{j=0}^{+\infty} a_j^{1}(\nu, \theta)
    \left(\psi_{(j,\nu)}(t) - \psi_{(j,\nu)}(-t)\right)\,,
  \end{eqnarray}

  with $\psi_{(j,\nu)}(t) - \psi_{(j,\nu)}(-t) = (1-(-1)^j)
  (t/\sqrt{\nu+t^2})^{-j}$. Therefore, as $a_j^{1}(\nu, \theta)>0$, \eqref{eq:sum-derivative-lr-student} yields
  $(\ilr)'(t) + (\ilr)'(-t) \geq 0$, which concludes the proof.
\end{proof}

\section{Convergence rate of a kernel estimator based on Storey's estimator}
\label{sec:conv-rate-kernel}
\begin{proof}[Proof of Proposition~\ref{prop:asy-var-pi0h}]
  \begin{enumerate}
  \item 
 We demonstrate that $\pihSto(1-h_m)$ may be written as a sum of $m$ independent random variables that satisfy the Lindeberg-Feller conditions for the Central Limit Theorem \citep{pollard84csp}. 
Let $Z^m_i = \ind{P_i \geq 1-h_m}$, where the $P_i$ are the $p$-values. $Z^m_i$ follows a Bernoulli distribution with parameter $1- \pcdf(1-h_m)$. Letting
\begin{displaymath}
  Y^m_i = \frac{Z^m_i-\Exp{Z^m_i}}{\sqrt{m h_m}}\,,
\end{displaymath}
we have $\sum_{i=1}^m Y^m_i = \sqrt{m h_m} \left(\pihSto(1-h_m)-\Exp{\pihSto(1-h_m)}\right)$.
The $(Y^m_i)_{1\leq i \leq m}$ are centered, independent random variables, with $\Var Y^m_i=\Var Z^m_i/(m h_m) = \pcdf(1-h_m)(1-\pcdf(1-h_m))/(m h_m)$, which is equivalent to $\ppdf(1)/m$ as $m \to +\infty$. Therefore,
\begin{displaymath}
  \lim_{m \to +\infty}  \sum_{i=1}^m \Exp{(Y^m_i)^2} = \ppdf(1)\,. 
\end{displaymath}
Finally we prove that for any $\eps>0$, 
\begin{displaymath}
  \lim_{m \to +\infty}  \sum_{i=1}^m \Exp{(Y^m_i)^2 \ind{\vert Y^m_i\vert > \eps}} = 0\,. 
\end{displaymath}
As $Z^m_i \in \{0,1\}$ and $\Exp{Z^m_i} \in [0,1]$, we have $(Y^m_i)^2\leq 1/(m h_m)$, and 
\begin{eqnarray*}
\sum_{i=1}^m \Exp{(Y^m_i)^2 \ind{\vert Y^m_i\vert > \eps}} & \leq & \frac{1}{h_m} \Exp{\ind{\vert Y^m_1\vert > \eps}}\\
& = & \frac{1}{h_m} \Prob{\left(\vert Y^m_1\vert > \eps\right)}\\
& \leq & \frac{1}{h_m} \frac{\Var Y^m_1}{\eps^2}
\end{eqnarray*}
by Chebycheff's inequality. As  $m h_m \to +\infty$ and $\Var Y^m_1 \sim \ppdf(1)/m$ as $m \to +\infty$, the above sum therefore goes to $0$ as $m h_m \to +\infty$. The Lindeberg-Feller conditions for the Central Limit Theorem are thus fulfilled, and we have 
\begin{displaymath}
    \sum_{i=1}^m Y^m_i \rightsquigarrow \mathcal{N}(0, \ppdf(1))\,,
\end{displaymath}
which concludes the proof.
\item As $\pcdf(\lambda) = \pi_0 \lambda + (1-\pi_0) \pcdfa(\lambda)$, we have, for any $\lambda <1$,
\begin{eqnarray}
  \label{eq:storey-alt}
  \frac{1-\pcdf(\lambda)}{1-\lambda} = \pi_0 + (1-\pi_0)\frac{1-\pcdfa(\lambda)}{1-\lambda}\,.
\end{eqnarray}
Therefore, the bias is given by
\begin{displaymath}
  \Exp{\pihl} - \pi_0 =  (1-\pi_0)\frac{1-\pcdfa(\lambda)}{1-\lambda}\,.
  \end{displaymath}

A Taylor expansion as $\lambda \to 1$ yields
  \begin{eqnarray*}
    1-\pcdfa(\lambda)  & = & \sum_{l=0}^{k}\frac{(-1)^{l}\ppdfa^{(l)}(1)}{(l+1)!}(1-\lambda)^{l+1} + \po{(1-\lambda)^{l+1}}\\
& = & (1-\lambda) \ppdfa(1) + \frac{(-1)^{k}\ppdfa^{(k)}(1)}{(k+1)!}(1-\lambda)^{k+1} + \po{(1-\lambda)^{k+1}}
  \end{eqnarray*}
as $\ppdfa^{(l)}(1) = (1-\pi_0)^{-1}\ppdf^{(l)}(1) = 0$ for $1\leq l < k$.  Therefore, if $h_m \to 0$ as $m \to +\infty$, we have
 \begin{displaymath}
    \Exp{\pihSto(1-h_m)} - \ppdf(1) = (1-\pi_0) \frac{(-1)^{k} \ppdfa^{(k)}(1)}{(k+1)!}h_m^k + \po{h_m^k}\,,
  \end{displaymath}
\end{enumerate}
which concludes the proof, as $(1-\pi_0)\ppdfa^{(k)}(1)=\ppdf^{(k)}(1)$. 
\end{proof}

\begin{proof}[Proof of Proposition~\ref{prop:opt-bw-sto}]
By Proposition~\ref{prop:asy-var-pi0h}, the asymptotic variance of $\pih(1-h_m)$ is equivalent to $\ppdf(1)/(m h_m)$, and the bias is of order $h_m^k$.
  The optimal bandwidth is obtained for $h_m$ proportional to $m^{-1/(2k+1)}$, because this choice balances variance and squared bias. The proportionality constant is an explicit function of $k$, $\pi_0$, $\ppdfa(1)$, and $\ppdfa^{(k)}(1)$. 
  By definition, the MSE that corresponds to this optimal choice is twice the corresponding squared bias, i.e. of order $m^{-2k/(2k+1)}$, which completes the proof of (1).
To prove (2), we note that 
\begin{displaymath}
    \sqrt{m h_m} \left(\pih - \ppdf(1)\right) =   \sqrt{m h_m} \left( \pih - \Exp{\pih}\right) + \sqrt{m h_m}\left( \Exp{\pih} - \ppdf(1)\right)\,,
\end{displaymath}
where $\pih$ denotes $\pih(1-h_m)$ to alleviate notation. The first term (variance) converges in distribution to $\mathcal{N}(0, \ppdf(1))$ by Proposition~\ref{prop:asy-var-pi0h} (1) as soon as $\sqrt{m h_m} \to +\infty$. The second term (bias) is of the order of $\sqrt{m h_m} h_m^k = \sqrt{m h_m^{2k+1}}$ by Proposition~\ref{prop:asy-var-pi0h} (2). Taking $h_m(k)=h_m ^\star (k) \eta^2_m$, where $\eta_m \to 0$, we have $m h_m^{2k+1} \to 0$, which ensures that the bias term converges in probability to 0.
\end{proof}

\section{Extension of Neuvial (2008) to the unconditional setting}
\label{sec:extens-uncond-setting}
In this section, we show that the results obtained by \cite{neuvial08asymptotic} in the original (conditional) setting of \cite{benjamini95controlling} also hold in the unconditional setting considered here, at the price of an additional term in the asymptotic variance due to the fluctuations of the random variable $\pi_{0,m}$.  We start by stating a lemma which provides a lower bound on the critical value of plug-in procedures.  It is is a consequence of Proposition~\ref{prop:criticality-BH}(1).

\begin{lemma} \label{lm:critic-plug-in}
Let $\alpha_m$ be a sequence of (possibly data-dependent) levels that converges in probability to $\alpha_{\infty} \in (0,1)$ as $m \to +\infty$.  
If $\alpha_\infty < \as_{BH}$, then the threshold $\tBH(\alpha_m)$ of the BH$(\alpha_m)$ procedure converges in probability to 0 as $m \to +\infty$.
If the convergence of $\alpha_m$ to $\alpha_\infty$ holds almost surely, then the convergence of $\tBH(\alpha_m)$ to 0 holds almost surely as well. 
\end{lemma}

\begin{proof}[Proof of Lemma~\ref{lm:critic-plug-in}]
  Assume that $\alpha_m$ converges to $\alpha_\infty$ in
  probability, with $\alpha_\infty < \as_{BH}$. Let $\eps > 0$, we are
  going to show that there exists an integer $N>0$ such that for a
  large enough $m$, the number of rejections of the BH$(\alpha_m)$
  procedure is less than $N$ with probability greater than $1-\eps$.
  Let $\bar{\alpha} = (\alpha_\infty + \as_{BH})/2$.  As $\alpha_m
  \overset{P}{\to} \alpha_\infty < \bar{\alpha}$, there exists an
  integer $M$ such that for any $m \geq M$, $\alpha_m \leq
  \bar{\alpha}$ with probability greater than $1-\eps/2$.  As
  $\bar{\alpha} < \as_{BH}$, Proposition~\ref{prop:criticality-BH}(1)
  entails that the number of rejections by the BH$(\bar{\alpha})$
  procedure is bounded in probability as $m \to +\infty$; that is,
  there exist two integers $N$ and $M'$ such that for $m \geq M'$,
  the number of rejections of the BH$(\bar{\alpha})$ procedure is
  less than $N$ with probability greater that $1-\eps/2$.  Thus, for
  any $m\geq \max(M,M')$, the number of rejections of the
  BH$(\alpha_m)$ procedure is less than $N$ with probability greater
  that $1-\eps$.
The proof for the almost sure convergence in the case when $\alpha_m$ converges to $\alpha_\infty$ almost surely is similar.
\end{proof}

We follow the proof technique introduced by \cite{neuvial08asymptotic}, by writing the empirical threshold of a given FDR controlling procedure (and its associated FDP) as the result of the application of a \emph{threshold function} of the empirical distribution of the observed $p$-values. As the regularity of the threshold functions involved has already been established by \cite{neuvial08asymptotic}, the result is a consequence of the fact that the $p$-value distributions under the null and the alternative hypotheses (as defined below) satisfy Donsker's theorem in the current unconditional setting.   This Donsker's theorem has been established by~\cite{genovese04a-stochastic}.
For $a \in \{0,1\}$ and $t \in [0,1]$, we let $\pecdfxU{a}(t) = m^{-1} \sum_{i=1}^m \ind{H_a\textrm{ true and } P_i\leq t}$.
\begin{proposition}[\cite{genovese04a-stochastic}, Theorem 4.1] As $m \to +\infty$, we have:
\label{prop:gw04}
\begin{enumerate}
\item 
  \begin{equation}
    \label{eq:gw04}
    \sqrt{m} \left(
      \begin{pmatrix}
        \pecdfnU(t) \\ \pecdfaU(t)
      \end{pmatrix} -
      \begin{pmatrix}
        \pi_0t \\ (1-\pi_0)\ppdfa(t)
      \end{pmatrix}
\right) \rightsquigarrow
\begin{pmatrix}
  \mathbb{W}_0 \\ \mathbb{W}_1
\end{pmatrix}
\,,
\end{equation}
where $\left(\mathbb{W}_0, \mathbb{W}_1\right)$ is a two-dimensional, centered Gaussian process with covariance function $\gamma(s,t)$ defined for any $(s,t) \in [0,1]^2$ by
\begin{equation}
  \label{eq:gw04-cov}
\gamma(s,t) = 
  \begin{pmatrix}
    \pi_0 s \wedge t - \pi_0^2 s t & -\pi_0 s (1-\pi_0) \pcdfa(t)\\
    -\pi_0 t (1-\pi_0) \pcdfa(s)  & (1-\pi_0) \pcdfa(s \wedge t) - (1-\pi_0)^2 \pcdfa(s) \pcdfa(t)
  \end{pmatrix}
\end{equation}
\item 
  \begin{equation}
\label{eq:gw04-G}
\sqrt{m}\left(\pecdf - \pcdf\right)  \rightsquigarrow \mathbb{W}\,,
\end{equation}
where $\mathbb{W} \overset{(d)}{=}\mathbb{W}_0+\mathbb{W}_1$ is a one-dimensional, centered Gaussian process with covariance function $(s,t) \mapsto \pcdf(s \wedge t) - \pcdf(s) \pcdf(t)$.
\end{enumerate}
\end{proposition}

Note that $\pecdfnU = \pi_{0,m} \pecdfn$ and $\pecdfaU = (1-\pi_{0,m}) \pecdfa$, where $(\pecdfn, \pecdfa)$ are the empirical distribution functions of the $p$-values under $\hn$ and $\ha$, respectively.  The results of \cite{neuvial08asymptotic} have been obtained by directly considering the convergence of the process  $(\pecdfn, \pecdfa)$ instead of $(\pecdfnU, \pecdfaU)$, because $\pi_{0,m} $ was deterministic in the conditional setting (see \citet[Theorem 3.1]{neuvial09corrigendum}).  The results established in \citet{neuvial08asymptotic} (in particular Theorem 3.2) can be translated to the unconditional setting just by replacing the processes $\pi_0 \mathbb{Z}_0$ and $\pi_1 \mathbb{Z}_1$ in \citet{neuvial08asymptotic} by the processes $\mathbb{W}_0$ and $\mathbb{W}_1$ defined in Proposition~\ref{prop:gw04}, and consequently, the process $\mathbb{Z} = \pi_0 \mathbb{Z}_0 + \pi_1 \mathbb{Z}_1$ by $\mathbb{W} = \mathbb{W}_0 + \mathbb{W}_1$.

Therefore, the asymptotic properties of the BH procedure and Storey's procedure (i.e. BH$(\cdot/\pihSto(\lambda))$ in the unconditional setting can be obtained by adapting the proof of the corresponding theorems (Theorems 4.2 and 4.15) in  \citet{neuvial08asymptotic}:

\begin{corollary}[Asymptotic properties of the BH procedure in the unconditional setting]
\label{cor:asy-prop-BH}
  For any $\alpha\geq \as_{BH}$, we have
  \begin{enumerate}
  \item The asymptotic distribution of the threshold $\tBH(\alpha)$ is given by 
    \begin{align}
    \label{eq:t-bh}
    \sqrt{m}\left(\tBH(\alpha)-\tBHinf(\alpha)\right) &\rightsquigarrow \mathcal{N}\left(0, \frac{\pcdf(\tBHinf(\alpha))(1-\pcdf(\tBHinf(\alpha)))}{(1/\alpha-\ppdf(\tBHinf(\alpha)))^2} \right)
  \end{align}
  \item The asymptotic distribution of the associated FDPs is given by
  \begin{align}
    \label{eq:fdp-bh}
    \sqrt{m}\left(\FDP_m(\tBH(\alpha))-\pi_0\alpha\right) &\rightsquigarrow \mathcal{N}\left(0,(\pi_0 \alpha)^2 \left(\frac{1}{\pi_0\tBHinf(\alpha)}-1\right)\right)    
  \end{align}
    \end{enumerate}
\end{corollary}
The asymptotic properties of the BH Oracle procedure are simply obtained by applying Corollary \ref{cor:asy-prop-BH} at level $\alpha/\pi_0$. 

\begin{corollary}[Asymptotic properties of Storey's procedure in the unconditional model]
\label{cor:asy-prop-Sto}
  For any $\lambda \in [0,1)$, and $\alpha \in [0,1]$, let $\tSto(\alpha) = \mathcal{T}^{\mathrm{Sto}(\lambda)}(\pecdf)$ be the empirical threshold $\tSto(\alpha)$ of Storey's procedure at level $\alpha$, and $\tStoinf(\alpha) = \mathcal{T}^{\mathrm{Sto}(\lambda)}(\pcdf)$ be the corresponding asymptotic threshold.
  Then,  
  \begin{enumerate}
  \item $\as_{\textnormal{Sto}(\lambda)} = \pibSto \as_{BH}$ is the critical value of Storey's procedure;
  \item For any $\alpha > \as_{\textnormal{Sto}(\lambda)} $:
  \begin{enumerate}
  \item The asymptotic distribution of the threshold $\tSto(\alpha)$ is given by 
\begin{align}
      \label{eq:t-bh} 
      \sqrt{m}\left(\tSto(\alpha)-\tStoinf(\alpha)\right) &\rightsquigarrow  \frac{\tStoinf(\alpha)}{\pibSto/\alpha - \ppdf(\tStoinf(\alpha))} \left\{
      \frac{\mathbb{W}(\tStoinf(\alpha))}{\tStoinf(\alpha)} + \frac{1}{\alpha} \frac{\mathbb{W}(\lambda)}{1-\lambda}
\right\}\,,
    \end{align}
    where
    $\mathbb{W}$ is a centered Gaussian process with covariance function
    $(s,t) \mapsto \pcdf(s \wedge t) - \pcdf(s) \pcdf(t)$;
  \item The asymptotic distribution of the associated FDPs is given by
    \begin{align}
      \label{eq:fdp-bh}
      \sqrt{m}\left(\FDP_m(\tSto(\alpha))-\pi_0\alpha/\pibSto\right) &\rightsquigarrow \mathcal{N}\left(0, \sigma_\lambda^2\right) \,,
    \end{align}
where
    \begin{displaymath}
      \sigma_\lambda^2 = \left(\frac{\pi_0\alpha}{\pibSto}\right)^2 \left\{
        \frac{1}{\pi_0\tStoinf(\alpha)} + 2\frac{\tStoinf(\alpha) \wedge \lambda}{\tStoinf(\alpha)(1-\pcdf(\lambda))} - \frac{1}{1-\pcdf(\lambda)} 
\right\}
    \end{displaymath}
  \end{enumerate}
  \end{enumerate}
\end{corollary}
Note that Corollary~\ref{cor:asy-prop-Sto} with $\lambda=0$ recovers Corollary~\ref{cor:asy-prop-BH}.

\section{Asymptotic properties of plug-in procedures}
\label{sec:asympt-prop-plug-in}

\subsection{Proof of Theorem~\ref{thm:asy-prop-plug-in}}
We denote by $\rPI(\alpha)$  the proportion of rejections, and by $\vPI(\alpha)$ the proportion of incorrect rejections by the plug-in procedure BH$(\alpha/\pih)$ (among all $m$ hypotheses tested).  They may be written as $\rPI(\alpha)=\pecdf(\tPI(\alpha))= \tPI(\alpha)\pih/\alpha$ and $\vPI(\alpha)=\pi_{0,m} \pecdfn(\tPI(\alpha))$, respectively.  
The following Lemma shows that the convergence rate of ($\tPI(\alpha), \vPI(\alpha), \rPI(\alpha))$ for a large enough $\alpha$ is driven by the convergence rate of $\pih$.  In order to alleviate notation, we omit the ``$(\alpha)$'' in $\tPI$, $\rPI$, $\vPI$, $\tPIinf$, $\rPIinf$, $\vPIinf$ in the remainder of this section.  Moreover,  $\fdp_m(\tPI(\alpha))$ will simply be denoted by $\fdpPI$.

\begin{lemma}
\label{lem:t-v-r-plug-in} Let $\pih$ be an estimator of $\pi_0$ such that $\pih \to \pi_{0,\infty}$ in probability as $m \to +\infty$.  Define $\asPI=\pi_{0,\infty}\as_{BH}$, and let $\alpha>\asPI$.  Then, under~\eqref{cond:concavity}, we have, as $m \to +\infty$:
\begin{enumerate}
\item $\tPI$ converges in probability to $\tPIinf$ as $m \to +\infty$, with $\ppdf(\tPIinf) < \pi_{0,\infty}/\alpha$.  If the convergence of $\pih$ to $\pi_{0,\infty}$ holds almost surely, then that of $\tPI$ to $\tPIinf$ holds almost surely as well;
\item Further assume that $\sqrt{m h_m} \left(\pih - \pi_{0,\infty}\right)$ converges in distribution for some $h_m$ such that $h_m= \po{1/\ln\ln m}$ and $mh_m \to +\infty$ as $m \to +\infty$.  Then $(\tPI, \vPI, \rPI)$ converges at in distribution at rate $1/\sqrt{m h_m}$, with 
  \begin{displaymath}
    \begin{pmatrix}
     \tPI\\ \vPI\\\rPI
    \end{pmatrix}
-     \begin{pmatrix}
     \tPIinf\\ \vPIinf\\ \rPIinf
    \end{pmatrix}
    = \frac{\tPIinf/\alpha}{\pi_{0,\infty}/\alpha-\ppdf(\tPIinf)}
    \begin{pmatrix}
      1 \\ \pi_0 \\ \ppdf(\tPIinf)
    \end{pmatrix}
    (\pi_{0,\infty}-\pih)(1+\pop{1}) \,,
  \end{displaymath}
where $\vPIinf= \pi_0\tPIinf$ and $\rPIinf= \pcdf(\tPIinf) = \pi_{0,\infty}\tPIinf/\alpha$.
\end{enumerate}
\end{lemma}

\begin{proof}[Proof of Lemma~\ref{lem:t-v-r-plug-in}]
For 1., we assume that the convergence of $\pih$ to $\pi_{0,\infty}$ holds in probability.  If it also holds almost surely, then the convergence of $\tPI$ to $\tPIinf$ is almost sure as well.  The sketch of the proof is inspired by \citet[Lemma 21.3]{vaart98asymptotic}.
  Let $\psi_{F, \zeta}:t\mapsto t/\zeta - F(t)$ for any distribution function $F$ and any $\zeta \in (0,1]$. As $\pecdf(\tPI)=\pih \tPI/\alpha$ and $\pcdf(\tPIinf)=\pi_{0,\infty} \tPIinf/\alpha$, we have $\psi_{\pcdf, \alpha/\pi_{0,\infty}}(\tPIinf)=0$ and $\psi_{\pecdf, \alpha/\pih}(\tPI)=0$. The proof relies on the following property:
\begin{description}
\item[(a)] $\psi_{\pcdf, \alpha/\pi_{0,\infty}}(\tPI)$ converges in probability to $0=\psi_{\pcdf, \alpha/\pi_{0,\infty}}(\tPIinf)$;
\item[(b)] $\psi_{\pcdf, \alpha/\pi_{0,\infty}}$ is locally invertible in a neighborhood of $\tPIinf$, with  $\dot{\psi}_{\pcdf, \alpha/\pi_{0,\infty}}(\tPIinf)>0$.
\end{description}
To prove $(a)$, we note that
\begin{eqnarray*}
  -\psi_{\pcdf,\alpha/\pi_{0,\infty}}(\tPI) & = & \pcdf(\tPI)-\pi_{0,\infty}\tPI/\alpha \\
  & = & (\pcdf-\pecdf)(\tPI) + (\pecdf(\tPI)-\pih \tPI/\alpha)+ (\pih-\pi_{0,\infty})\tPI/\alpha\,.
\end{eqnarray*}
The first term converges to $0$ almost surely, the second one is identically null, and the third one converges in probability to $0$ as $\pih$ converges in probability to $\pi_{0,\infty}$, and $\tPI \in [0,1]$. Item (b) holds as $\pcdf$ in concave (by~\eqref{cond:concavity}) and $\alpha / \pi_{0,\infty} > \as_{BH}$, where $\as_{BH}=\lim_{u \to 0}u/\pcdf(u)$ is the critical value of the BH procedure (see \citet[Lemma 7.6 page 1097]{neuvial08asymptotic} for a proof of the invertibility). 
\begin{enumerate}
\item  Combining $(a)$ and $(b)$, $\tPI$ converges in probability to $\tPIinf$, and $\dot{\psi}_{\pcdf, \alpha/\pi_{0,\infty}}(\tPIinf) = \pi_{0,\infty}/\alpha - \ppdf(\tPIinf) $  is positive.
\item We only give the proof for $\tPI$, as the proofs for $\vPI$ and $\rPI$ are similar.
The idea of the proof is that the fluctuations of $\prcdf=\pecdf-\pcdf$, the centered empirical process associated with $\pcdf$, are of order $1/\sqrt{m}$ by Donsker's theorem~\citep{donsker51an-invariance}; thus, these fluctuations are negligible with respect to the fluctuations of $\pih-\pi_{0,\infty}$, which are assumed to be of order $1/\sqrt{m h_m}$ with $h_m \to 0$. We have 
\begin{eqnarray*}
  \pcdf(\tPI) - \pcdf(\tPIinf) & = & (\pcdf(\tPI)-\pecdf(\tPI)) + (\pecdf(\tPI) - \pcdf(\tPIinf))\\
  & = & -\prcdf(\tPI) + (\pih \tPI/\alpha - \pi_{0,\infty} \tPIinf/\alpha)
\end{eqnarray*}
because $\pecdf(\tPI)=\pih \tPI/\alpha$ and $\pcdf(\tPIinf)=\pi_{0,\infty} \tPIinf/\alpha$. Therefore,
\begin{displaymath}
  \pcdf(\tPI) - \pcdf(\tPIinf)  =  -\prcdf(\tPI) + \frac{\pih}{\alpha}(\tPI-\tPIinf)+ \frac{\pih-\pi_{0,\infty}}{\alpha}\tPIinf\,.
\end{displaymath}  \label{the-proof}
As $\tPI \overset{P}{\to} \tPIinf$ as $m \to +\infty$, we also have $\pcdf(\tPI) - \pcdf(\tPIinf) = (\tPI-\tPIinf)(\ppdf(\tPIinf)+\pop{1})$ by Taylor's formula. Hence we have

\begin{displaymath}
 \left (\ppdf(\tPIinf)-\pih/\alpha + \pop{1}\right)(\tPI-\tPIinf)  =  -\prcdf(\tPI) + (\pih-\pi_{0,\infty})\tPIinf/\alpha\,.
  \end{displaymath}
  Now because $\pih$ converges in probability to $\pi_{0,\infty}$, we have $\ppdf(\tPIinf)-\pih/\alpha = (\ppdf(\tPIinf)-\pi_{0,\infty}/\alpha)(1+ \pop{1})$.  By 1, we have $\pi_{0,\infty}/\alpha > \ppdf(\tPIinf)$, so that for sufficiently large $m$:
\begin{displaymath}
 \tPI-\tPIinf  =  \frac{\prcdf(\tPI)}{\ppdf(\tPIinf)-\pi_{0,\infty}/\alpha} (1+\pop{1})+ \frac{\tPIinf/\alpha}{\ppdf(\tPIinf)-\pi_{0,\infty}/\alpha}(\pih-\pi_{0,\infty})\,.
  \end{displaymath}
Finally, we note that as $\Vert \prcdf \Vert_\infty \sim c\sqrt{\ln\ln m/m}$ (by the Law of the Iterated Logarithm) and $h_m = \po{1/\ln\ln m}$, we have $\prcdf(\tPI) = \pop{1/\sqrt{m h_m}}$.  On the other hand,  $\sqrt{m h_m} \left(\pih - \pi_{0,\infty}\right)$ converges in distribution, so that the term $ (\pih-\pi_{0,\infty})\tPIinf/\alpha$ dominates the right-hand side.  Finally, we have
  \begin{displaymath}
    \tPI - \tPIinf   = \frac{\tPIinf/\alpha}{\ppdf(\tPIinf)-\pi_{0,\infty}/\alpha} (\pih-\pi_{0,\infty})(1+\pop{1}) \,,
  \end{displaymath}
which concludes the proof for $\tPI$. 
  \end{enumerate}\vspace{-4ex}
\end{proof}

\begin{proof}[Proof of Theorem~\ref{thm:asy-prop-plug-in}]
1. is a consequence of Lemma~\ref{lm:critic-plug-in} combined with Lemma~\ref{lem:t-v-r-plug-in}(1); 2.(a) is a consequence of Lemma~\ref{lm:critic-plug-in}(2). 
Let us prove 2.(b). By Lemma~\ref{lem:t-v-r-plug-in}, we have
    \begin{equation}
      \label{eq:t-v-r-plug-in}
      \sqrt{m h_m} \left(
        \begin{pmatrix}
          \vPI\\ \rPI
        \end{pmatrix}
        - 
        \begin{pmatrix}
          \vPIinf \\ \rPIinf
        \end{pmatrix} \right) \rightsquigarrow \xi_{\infty}
      \begin{pmatrix}
        \pi_0 \\ \ppdf(\tPIinf)
      \end{pmatrix}
      X \,, 
    \end{equation}
    where $X \sim \mathcal{N}(0, s_0^2)$ and \begin{displaymath}
      \xi_\infty = \frac{\tPIinf
        /\alpha}{\pi_{0,\infty}/\alpha-\ppdf(\tPIinf)} \,.
    \end{displaymath}
    Recall that $\fdpPI = \vPI/(\rPI \vee m^{-1})$.  We begin by
    noting that for a large enough $m$, we have $\rPI>1/m$ almost
    surely.  This is a consequence of the fact that (i) $\rPI =
    \pecdf(\tPI) = \pih \tPI/\alpha$, with $\tPI$ bounded away from 0
    (by 1.), and (ii) $\pih$ converges to $\pi_{0,\infty} \geq \pi_0 >
    \alpha$.
    As a consequence, the factor $ m^{-1}$ may be omitted in $\fdpPI$
    for a large enough $m$;
    the FDP may then be written as $\fdpPI=\gamma(\vPI,\rPI)$, where
    $\gamma:(u, v) \mapsto u/v$ for any $u\geq 0$ and $v >
    0$. $\gamma$ is differentiable for any such $(u,v)$, with
    derivative $\dot{\gamma}_{u,v} = (1/v, -u/v^2) = 1/v (1, -u/v)$.
    In particular, recalling that $\vPIinf= \pi_0\tPIinf$ and $\rPIinf
    = \pcdf(\tPIinf) = \pi_{0,\infty}\tPIinf/\alpha$, we have
    \begin{equation}\label{eq:dot-gamma}
      \dot{\gamma}_{\vPIinf, \rPIinf} = \frac{\alpha}{\tPIinf \pi_{0,\infty}} \left(1, -\frac{\pi_0 \alpha}{\pi_{0,\infty}}\right)\,.
    \end{equation}
    As $\gamma(\vPIinf, \rPIinf)=\pi_0\alpha/\pi_{0,\infty}$, the
    Delta method yields
    \begin{displaymath}
      \sqrt{mh_m}\left(\fdpPI-\frac{\pi_0\alpha}{\pi_{0,\infty}}\right) \rightsquigarrow \mathcal{N}\left(0, w^2\right)\,,
    \end{displaymath} 
    with $w = s_0 \xi_\infty~\dot{\gamma}_{\vPIinf, \rPIinf}
    \begin{pmatrix}
      \pi_0\\ \ppdf(\tPIinf)
    \end{pmatrix}
    $.

    By (\ref{eq:dot-gamma}), we have $ \dot{\gamma}_{\vPIinf, \rPIinf}
    \begin{pmatrix}
      \pi_0\\ \ppdf(\tPIinf)
    \end{pmatrix}
    % = \frac{\alpha}{\tPIinf \pi_{0,\infty}} (\pi_0 - \alpha
    % \ppdf(\tPIinf)/\pi_{0, \infty})\\
    = \frac{\alpha^2 \pi_0}{\tPIinf \pi_{0,\infty}^2}
    (\pi_{0,\infty}/\alpha - \ppdf(\tPIinf))$, so that $w = s_0
    \pi_0\alpha/\pi_{0, \infty}^2 $.
\end{proof}

\subsection{Consistency, purity and criticality}
\label{sec:consistency-purity-criticality}
\begin{proof}[Proof of Lemma~\ref{lm:purity-critic}]
We note that
\begin{eqnarray*}
\frac{\tpdfa(x)}{\tpdfn(x)} & = & \frac{\tpdfn(x-\theta)}{\tpdfn(x)}  \textrm{ by definition of a location model} \\
&  = & \frac{\tpdfn(-x+\theta)}{\tpdfn(-x)}  \textrm{ by~\eqref{cond:symmetry}}\\
& = & \frac{\tpdfn(-x+\theta)}{\tpdfa(-x+\theta)}\,,\ \ 
\end{eqnarray*}
which concludes the proof, as $\theta$ is a fixed scalar.
\end{proof}

\begin{proof}[Proof of Proposition~\ref{prop:purity-critic}]
We have $\as_{BH} = \lim_{t \to 0} 1/\ppdf(t)$, where $\ppdf = \pi_0 + (1-\pi_0)\ppdfa$ and 
\begin{displaymath}
      \ppdfa(t) = \lr\left(-\tcdfn^{-1}(t)\right)\,.
\end{displaymath}
Therefore, as $\lim_{t \to 0} \tcdfn^{-1}(t) = +\infty$, the result is a consequence of Lemma~\ref{lm:purity-critic}.
\end{proof}

\subsection{Regularity of $\ppdfa$ for two-sided tests in symmetric models}
\label{sec:regularity-two-sided}\begin{proof}[Proof of Lemma~\ref{lm:regularity-two-sided-p}]
  \begin{enumerate}
  \item We  make the additional assumption that there exists $\eta > 0$ such that $\ilr$ is differentiable on $V_\eta=[-\eta, \eta] \setminus \{0\}$, and that its derivative tends to $\ell_-$ as $u \to 0^-$ and  $\ell_+$ as $u \to 0^+$.  This assumption makes the proof simpler, and it holds in the models considered in this paper. However, the result still holds (and is simpler to state) without this extra assumption.
  By Proposition~\ref{prop:p-values}, we have under~\eqref{cond:symmetry}
  \begin{displaymath}
      \ppdfa(t) = \frac{1}{2}\left(\lr (q_0(t/2)) + \lr (-q_0(t/2)) \right)\,,
    \end{displaymath}
    where $q_0(t/2) = \tcdfn^{-1}(1-t/2)$ maps  $Q_\eta = [2(1-\tcdfn(\eta)), 1)$ onto  $(0, \eta]$. Therefore, $\ppdfats$ is differentiable on $Q_\eta$ and satisfies, for any $t$ in $Q_\eta$:
    \begin{eqnarray}
      \ppdfa^{ (1)}(t)  & = & \frac{1}{2} \left\{ \left(\lr \right)' (q_0(t/2)) - \left(\lr \right)' (-q_0(t/2)) \right\} 
      \times \frac{1}{2} q_0'(t/2) \notag \\
 & = & -\frac{1}{4\tpdfn(q_0(t/2))} \left(
    \left(\lr\right)' (q_0(t/2)) 
    - \left(\lr\right)' (-q_0(t/2))
  \right) \label{eq:ppdfa-derivative}
   \end{eqnarray}
As $t \to 1$, $q_0(t/2) \to 0^+$,~(\ref{eq:ppdfa-derivative}) implies that $\ppdfa$ is differentiable at 1 with derivative $-(4\tpdfn(0))^{-1}(\ell_+-\ell_-)$.
\item Similarly, we  prove the result with the extra assumption that $\ilr$ is twice differentiable in a neighborhood of 0.  Then~(\ref{eq:ppdfa-derivative}) entails that $\ppdfa$ is itself twice differentiable in a neighborhood of 1.  Writing $ \ppdfa^{(1)}(t) = a(t) b(t)$, with 
  \begin{displaymath}
    \begin{cases}
      a(t) & = 1/(4\tpdfn(q_0(t/2))) \\
      b(t) & = - \left(\ilr\right)' (q_0(t/2)) + \left(\ilr\right)' (-q_0(t/2) )
   \end{cases}\,,
  \end{displaymath}
we have $\ppdfa^{(2)}(t) =a'(t)b(t) + a(t)b'(t)$. As $q_0(1/2) = \tcdfn^{-1}(1/2)=0$ , we have $b(1)=0$, so that $\ppdfa^{(2)}(1) = a(1)b'(1)$, where  $a(1)=1/(4\tpdfn(0))$ and
\begin{displaymath}
      b'(t) = \frac{1}{2\tpdfn(q_0(t/2))} \left(\left(\lr\right)^{(2)} (q_0(t/2)) + \left(\lr\right) ^{(2)} (-q_0(t/2) )\right)\,.
  \end{displaymath}
Thus $b'(1)=1/(2\tpdfn(0)) \times 2(\ilr)^{(2)}(0)$, which concludes the proof.
  \end{enumerate}
\end{proof}

\addcontentsline{toc}{section}{References}
\bibliography{critic}

\end{document}